\documentclass{amsart}
\usepackage{graphicx,amsfonts,amssymb,amsmath,amsthm,url}

\theoremstyle{plain} %--default
\newtheorem{theorem}             {Theorem}  [section]
\newtheorem{lemma}      [theorem]{Lemma}
\newtheorem{corollary}  [theorem]{Corollary}

\theoremstyle{definition}

\newtheorem{example}    [theorem]{Example}

\theoremstyle{remark}
\newtheorem{remark}              {Remark}

\def\Gal{\operatorname{Gal}}
\def\sgn{\operatorname{sgn}}
\def\sym{\operatorname{sym}}
\def\Tr{\operatorname{Tr}}
\def\SL{\operatorname{SL}}
\def\GL{\operatorname{GL}}
\renewcommand{\Re}{\mathrm{Re}}
\renewcommand{\Im}{\mathrm{Im}}
\def\eps{\varepsilon}
\newcommand{\RestrictK}{$(k \geq 4)$ }
\def\otherN{N'}

\newcommand{\efrac}[2]{\genfrac{}{}{0pt}{}{#1}{#2}}

\begin{document}

\title[Stable averages
of central values of Rankin-Selberg L-functions]%
{Stable averages
  of central values of \\
  Rankin-Selberg L-functions: some new variants }

\author{Paul Nelson}
% \address{Department of Mathematics\\
%   California Institute of Technology\\
%   Pasadena, California 91125}%

% \subjclass{Primary 54C40, 14E20; Secondary 46E25, 20C20}
% \date     {July 2, 1991}
% \thanks{The first author was supported in part by NSF Grant \#000000.}
% \dedicatory{}
\begin{abstract}
  As shown by Michel-Ramakrishan (2007) and later generalized by
  Feigon-Whitehouse (2008), there are ``stable'' formulas for
  the average central L-value of the Rankin-Selberg
  convolutions of some holomorphic forms of fixed even weight and
  large level against a fixed imaginary quadratic theta series.
  We obtain exact finite formulas for twisted first moments of
  Rankin-Selberg L-values in much greater generality and prove
  analogous ``stable'' formulas when one considers either
  arbitrary modular twists of large prime power level or real
  dihedral twists of odd type associated to a Hecke character of
  mixed signature.
\end{abstract}
\maketitle

\tableofcontents

The special values $L(f,s)$ of $L$-functions attached to
automorphic forms $f$ are of fundamental arithmetic interest;
for instance, such values (often conjecturally) carry
information concerning the arithmetic of number fields (the
class number formula) and elliptic curves (the Birch and
Swinnerton-Dyer conjecture).  Motivated by this interest, a
basic problem in modern number theory is to study the behavior
of such values as $f$ traverses a family of automorphic forms.
Some typical problems of interest are to
\begin{enumerate}
\item show that $L(f,s)$ is non-vanishing for at
least one (or several) such $f$,
\item show that $L(f,s)$ satisfies a nontrivial upper bound in terms
  of $s$ and the conductor of $f$ (the subconvexity
  problem),
  and
\item study the (possibly twisted) moments of $f \mapsto
  L(f,s)$; such study has often served as technical input in
  approaches to the above two problems.
\end{enumerate}

In this paper we consider the family of Rankin-Selberg
$L$-values $L(f \otimes g, s)$ where $g$ is a fixed holomorphic
modular form on $\GL(2)/\mathbb{Q}$ and $f$ traverses a
family of holomorphic cusp forms of
fixed weight, level and nebentypus.  We are motivated by work of
Michel-Ramakrishnan \cite{MR07} and later Feigon-Whitehouse
\cite{FW08}, who show for certain dihedral forms $g$
arising from idele class characters on imaginary quadratic
fields that there
are \emph{finite} formulas for the twisted first moments of
central values $f \mapsto L(f \otimes g,\tfrac{1}{2})$ that
simplify considerably, reducing to just one or two terms, when
the level of the family to which $f$ belongs is taken
to be sufficiently
large.  One application of such finite formulas that would be
inaccessible with an inexact asymptotic formula is to
show that there exist
$f$ for which the algebraic part of $L(f \otimes
g,\tfrac{1}{2})$ is non-vanishing modulo a prime $p$.  The
present work stemmed from a desire to understand better the
scope and generality of such phenomena.  Because the methods of
\cite{MR07} and \cite{FW08} make essential use of the
restriction that $g$ be dihedral by invoking respectively the
Gross-Zagier formula and Waldspurger's formula, we wondered whether
the results obtained are likewise exclusive to dihedral $g$ or
if they extend to general modular forms $g$.  Our aim in this
paper is to show that they do in fact hold quite generally.

Before stating our own results, let us recall in more detail
the relevant results of Michel-Ramakrishnan \cite{MR07}:
\begin{enumerate}
\item[(I)] Let $-D$ be a negative odd fundamental discriminant, let $\Psi$ be
  a class group character of $\mathbb{Q}(\sqrt{-D})$, and let
  $g_\Psi$ be the weight $1$ theta series of level $D$ and nebentypus
  $\chi_{-D} = (-D |\cdot)$ attached to $\Psi$.
  Let $N$ be a rational prime that is inert in
  $\mathbb{Q}(\sqrt{-D})$,
  let $k$ be a positive even integer,
  and let $f$ traverse the set of arithmetically-normalized
  ($\lambda_f(1) = 1$)
  holomorphic newforms of weight $k$ on
  $\Gamma_0(N)$.
  Then there is a simple finite formula
  \cite[Thm 1]{MR07} for the
  twisted first moment of central $L$-values
  \begin{equation*}
    \sum_f \frac{L(f \otimes g_\Psi,\tfrac{1}{2}) }{
      \int_{\Gamma_0(N) \backslash \mathbb{H}} |f|^2
      y^{k} \, \frac{d x \, d y}{y^2}}
    \lambda_m(f),
  \end{equation*}
  where 
  $\lambda_m(f) m^{(k-1)/2}$ is the $m$th Fourier coefficient of $f$
  and $L(f \otimes g,s)$ is normalized so that it satisfies
  a functional
  equation
  under $s \mapsto 1-s$.
  We have spelled out the Petersson norm explicitly here because
  we shall
  use a different normalization later in the paper
  (see \eqref{eq:1}).
\item[(II)] Moreover, the formula in question
  becomes astonishingly simple in the so-called
  ``stable range'' $N > m D$, in which case all but one or two
  of its terms vanish;
  for instance, if $k \geq 4$ and $N > m D$ then one has
  \cite[Cor 1]{MR07}
  \begin{equation*}
    \frac{\Gamma(k-1)}{(4 \pi)^{k-1}}
    \sum_f \frac{L(f \otimes g_\Psi,\tfrac{1}{2}) }{
      \int_{\Gamma_0(N) \backslash \mathbb{H}} |f|^2
      y^{k} \, \frac{d x \, d y}{y^2}
      }
      \lambda_m(f)
      =   
    2 
    \frac{\lambda_m(g)}{m^{1/2}} L(\chi_{-D},1),
  \end{equation*}
  where $\lambda_m(g) m^{(l-1)/2}$ (with $\ell = 1$) is the $m$th Fourier coefficient
  of $g = g_\Psi$
  and $L(\chi_{-D},1) = \sum_{n \geq 1} \chi_{-D}(n)/n
  = 2 \pi h / (w D^{1/2})$ with $h$ the class number of
  $\mathbb{Q}(\sqrt{-D})$
  and $w$ the order of its unit group.
\end{enumerate}
As an application, the authors of \cite{MR07} derive some hybrid
subconvexity, non-vanishing and non-vanishing mod $p$ results
for $N$ and $D$ in certain ranges; while conceivably the subconvexity and
non-vanishing results could have also been derived with a
non-exact asymptotic formula having a $o(1)$ term, the
non-vanishing mod $p$ results relied crucially on the finiteness
of the formula.
Note also that while subconvex bounds for $L(f \otimes g,
\tfrac{1}{2})$
are known in generality
\cite{michel-2009} when either $f$ or $g$ is fixed,
the results of \cite{MR07} address the case that
$f$ and $g$ vary simultaneously while satisfying the constraint
$(k D)^\delta \ll N \ll D (k D)^{-\delta}$ for some fixed
$\delta > 0$.
Feigon-Whitehouse \cite{FW08} generalized many
of these results to the context of holomorphic Hilbert modular
forms of squarefree level $\mathfrak{N}$ over a totally real
number field $F$ averaged against a fixed theta series
associated to an idele class character on a CM extension $K/F$
under certain additional conditions such as that
$\mathfrak{N}$ be a squarefree product of primes
that are inert in $K$.

As indicated above, our aim in this paper is to prove analogues
of assertions (I) and (II)
for general (not necessarily dihedral)
holomorphic forms $g$ on $\GL(2)/\mathbb{Q}$.
To give a flavor for the results we obtain, we begin
by stating
one of the simplest (and to the author, most surprising) consequences.
Denote by $S_k(N,\chi)$ the space of
holomorphic cusp forms of
weight $k$,
level $N$ and nebentypus $\chi$,
and let $S_k(N) =S_k(N,\chi_0)$ where $\chi_0$ is the principal
character
of modulus $N$.
For $f = \sum a_n n^{(k-1)/2} q^n \in S_k(N,\chi)$
and $g = \sum b_n n^{(l-1)/2} q^n \in S_l(D,\eps)$
($q = \exp(2 \pi i z)$),
we define
\begin{equation}\label{eq:4}
  L(f \otimes g, s) = L(\chi \eps,2 s)
  \sum_{n \geq 1} \frac{a_n b_n}{n^s}
\end{equation}
for $\Re(s) > 1$ and in general by meromorphic continuation
\cite{Ra39,Se40,Ja72,Li79};
note that while this ad hoc definition has classical precedent
as in \cite{Li79}, it may differ by some bad Euler factors
from the canonical normalization when $f$ and $g$ are
newforms.
The critical line for
$L(f \otimes g, s)$
is $\Re(s) = 1/2$,
and $L(\chi \eps, s) = \sum_{n \geq 1} \chi(n) \eps(n) n^{-s}$
($\Re(s) > 1$) is the Dirichlet $L$-function;
note that if for instance $\chi_0$ is the principal character mod
$N$,
then $L(\chi_0 \eps,s)$ is $L(\eps,s)$ without the Euler factors
at $N$.
\begin{theorem}\label{thm:simple}
  Let $l$ be an even positive integer,
  let $k \geq 5$ be an odd positive integer
  such that $k > l$,
  and let $\chi$ be a primitive Dirichlet character of conductor
  $N$.
  Let $\mathcal{F}$ be an orthogonal basis of $S_k(N,\chi)$
  with respect to the Petersson inner product.
  Then for each fixed $g = \sum b_n n^{(l-1)/2} q^n \in S_l(1)$,
  we have
  \begin{equation}\label{eq:2}
    \frac{\Gamma(k-1)}{(4 \pi)^{k-1}}
    \sum_{f \in \mathcal{F}}
    \frac{
      L(f \otimes g,\tfrac{1}{2})
    }
    {
      \int_{\Gamma_0(N) \backslash \mathbb{H}} |f|^2
      y^{k} \, \frac{d x \, d y}{y^2}
    }
    = b_1 L(\chi,1).    
  \end{equation}
\end{theorem}
For example, Theorem \ref{thm:simple}
applies when $g$ is the modular discriminant
of weight $12$ and $f$ traverses any space of cusp forms
of weight $13$ and primitive nebentypus $\chi$.
It shows immediately that for some $f$ as above,
the algebraic part of $L(f \otimes g,\tfrac{1}{2})$
is non-vanishing modulo any prime of $\overline{\mathbb{Q}}$
for which that of $L(\chi,1)$ is non-vanishing.
If one postulates the nonnegativity
of $L(f  \otimes g, \tfrac{1}{2})$
then Theorem \ref{thm:simple} gives a hybrid subconvex bound
for $N,\chi,k,l$ as above
provided that $(k-l) \gg (k N)^{1+\delta}$ for some $\delta >
0$.

Our main findings may be summarized as follows.
\begin{itemize}
\item When $g \in S_l(D,\eps)$ is an arbitrary fixed holomorphic
  cusp form of squarefree level $D$, we obtain finite formulas
  along the lines of assertion (I) above for the twisted first
  moments of $S_k(N,\chi) \ni f \mapsto L(f \otimes g, s)$
  whenever
  $k > l$,
  $2 s$ is an integer of the same parity as $k \pm l$, and
  $1 - ( k-l)/2 \leq s \leq ( k-l)/2$ (Theorem
  \ref{thm:finite-formula}).
  Under the above conditions, $s$ is a critical value of $L(f \otimes g, s)$
  in the sense of Deligne.
  When $s = 1/2$, these condition hold if and only if
  $k > l$ and
  the parities of $k$ and $l$ are opposite.
\item We find
  that assertion (II) above extends without essential
  modification to arbitrary cusp forms $g$ with primitive
  nebentypus and squarefree level (Theorem
  \ref{thm:real-simplification}, Remark \ref{rmk:why-is-this-theorem-relevant}).
  For instance, while the
  methods of \cite{MR07} and \cite{FW08} apply only to dihedral
  forms coming from idele class characters on imaginary
  quadratic fields, we obtain analogous results for the
  (holomorphic) theta series attached to finite order mixed
  signature idele class characters on real quadratic fields
  (Theorem \ref{thm:real-concrete}).
\item Assertion (II) above says that the finite formula
  for the twisted first moment
  simplifies considerably when
  the level of
  the varying form $f$ is sufficiently large compared to that
  of the fixed form $g$.
  We observe a new phenomenon: such simplification occurs also when
  the level of $f$ is sufficiently divisible by prime divisors
  of the level of $g$ (Theorem \ref{thm:stab-vert-sense}).
  Note that in the works \cite{MR07}
  and \cite{FW08}, the levels of $f$ and $g$ are always taken
  to be relatively prime.
\end{itemize}
Our analysis makes use of the results and method of Goldfeld-Zhang
\cite{GoZh99}, who compute the kernel of the linear map $f
\mapsto L(f \otimes g, s)$ in some generality.  They suggest in
their paper that their results may have some applications for
special values of $L(f \otimes g,s)$, and we consider our work
in that spirit.

Our paper is organized as follows.
In \S\ref{sec:main-result}
we state our general result, which requires a fair amount of
notational baggage;
the reader is encouraged to skim this section on a first reading
and to look instead at
\S\ref{sec:examples}, in which some simple but representative
examples are spelled out.
In the remainder of \S\ref{sec:results}
we report on some numerical checks of our formulas
and describe some of the applications mentioned above.
In \S\ref{sec:proofs} we give proofs.

\section{Results}\label{sec:results}

\subsection{Main result}\label{sec:main-result}
In this section we state our main result, from which all others
shall follow; we spell out some special
cases that require less notational overhead in \S\ref{sec:examples}.
Throughout this paper we let
$k,l,N,D$ be positive integers
and 
$\chi$ mod $N$, $\eps$ mod
$D$ Dirichlet characters such that
\begin{equation}\label{eq:assumptions}
  \text{$k \geq 4$, $D$ is squarefree,
    $\chi(-1) = (-1)^k$,
    and $\eps(-1) = (-1)^l$.}
\end{equation}
We expect that our main results
(Theorem \ref{thm:eta-value} and its corollaries)
hold under the weaker condition $k \geq 2$
(the restriction to $k \geq 4$ is discussed in
\S\ref{sec:some-results-goldf-2}).

The group $\GL_2(\mathbb{R})^+$ acts by fractional linear
transformations
on the upper half-plane $\mathbb{H} = \{x + i y : y > 0\}$
in the usual way.
Recall the weight $k$ slash operator
on functions $f: \mathbb{H} \rightarrow \mathbb{C}$:
for $\alpha = \left(
  \begin{smallmatrix}
    a&b\\
    c&d
  \end{smallmatrix}
\right) \in \GL_2(\mathbb{R})^+$, the function $f|_k \alpha$
sends $z$ to $\det(\alpha)^{k/2} (c z + d)^{-k} f(\alpha z)$.
The space $S_k(N,\chi)$ of holomorphic cusp forms
of weight $k$, level $N$ and nebentypus $\chi$ consists of holomorphic
functions $f : \mathbb{H} \rightarrow \mathbb{C}$
that satisfy
\begin{equation*}
  f|_k \gamma = \chi(d) f
  \quad \text{ 
    for all }
  \gamma =
  \begin{bmatrix}
    * & * \\
    * & d
  \end{bmatrix}
  \in \Gamma_0(N) = \SL_2(\mathbb{Z}) \cap
  \begin{bmatrix}
    \mathbb{Z}  & \mathbb{Z}  \\
    N \mathbb{Z}  & \mathbb{Z} 
  \end{bmatrix}
\end{equation*}
and vanish at the cusps of $\Gamma_0(N)$;
the space $S_l(D,\eps)$ is defined analogously.
For $g \in S_l(D,\eps)$
we write
\begin{equation}
  g(z) = \sum_{m = 1}^\infty
  b_m m^{(l-1)/2} q^m,
  \quad q = e^{2 \pi i z},
\end{equation}
so that the Fourier coefficient $b_m$ of $g$ is normalized so
that the Deligne bound reads $|b_p/b_1| \leq 2$ when $g$ is a
newform.  The cusps of $\Gamma_0(D)$ are indexed by the
factorizations of $D$ as a product $D = \delta \delta '$ of
positive integers $\delta$ and $\delta '$.  The scaling matrix
for the cusp indexed by $\delta$ is
\begin{equation}
  W_\delta =
  \begin{bmatrix}
    * & * \\
    \delta  & \delta '
  \end{bmatrix}
  \begin{bmatrix}
    \delta  &  \\
    & 1
  \end{bmatrix}
  \quad \text{ with }
  \begin{bmatrix}
    * & * \\
    \delta  & \delta '
  \end{bmatrix} \in \SL_2(\mathbb{Z}).
\end{equation}
The matrix $W_\delta$ is uniquely determined up to
left-multiplication by $\Gamma_0(D)$.
We write
\begin{equation}
  g| W_\delta (z) =
  \sum_{m = 1}^{\infty}
  b_m^\delta m^{(l-1)/2} q^m
\end{equation}
for the Fourier coefficients
of $g$ at the cusp indexed by $\delta$,
so that in particular $b_m^1 = b_m$.
This notation for the Fourier coefficients of $g$
at the cusps of $\Gamma_0(D)$
will be in effect throughout the paper.
In the special case that $\eps$ is primitive
and $g$ is a normalized newform, Atkin-Li \cite{MR508986}
obtained a formula for the coefficients $b_n^D$,
which we collect here for convenience:
\begin{theorem}\label{thm:eta-value}
  Suppose that $\eps$ is primitive and that $g = \sum_{n \geq 1}
  b_n n^{\frac{l-1}{2}} q^n \in S_l(D,\eps)$ is a
  newform with $b_1 = 1$.
  Define the Gauss sum $\tau(\eps)$
  as in \eqref{eq:14}.
  Then $|b_D| = 1$
  and
  \begin{equation}
    b_n^D = \eps(-1) \frac{\tau(\eps)}{D^{1/2}} \overline{b_D
      b_n}
    \quad \text{ for all } n \geq 1.
  \end{equation}
\end{theorem}
\begin{proof}
  One has $W_D \in \Gamma_1(D) \left(
    \begin{smallmatrix}
      0  &-1\\
      D& 0
    \end{smallmatrix}
  \right)$, so the multiplicity-one theorem implies $b_n^D =
  \overline{\eta(g) b_n}$
  for some scalar $\eta(g)$; according to \cite[Thm
  6.29]{Iw97}, \cite[Prop 14.15]{MR2061214} and
  \cite[6.89]{Iw97}, we have $\eta(g) = \tau(\overline{\eps})
  b_D D^{-1/2}$.
  Using that $\overline{\tau(\overline{\eps})} =
  \eps(-1) \tau(\eps)$, we obtain the claimed formula.  Note that Atkin-Li
  actually consider the operator $h \mapsto h|(-W_D) = (-1)^l
  h|W_D$, so one must take care in citing their results.
\end{proof}

It will be convenient to define a scaled Petersson inner product
on $S_k(N,\chi)$ by the formula
\begin{equation}\label{eq:1}
  (f,g) =
  \frac{(4 \pi)^{k-1}}{\Gamma(k-1)}
  \int_{\Gamma_0(N) \backslash \mathbb{H}} \overline{f(z)} g(z)
  y^{k} \, \frac{d x \, d y}{y^2}.
\end{equation}
By abuse of notation, we write $\sum_{f \in S_k(N,\chi)}$ for the
sum
over all $f$ in any fixed orthogonal basis of $S_k(N,\chi)$
with respect to the inner product (\ref{eq:1}).

Recall the definition (\ref{eq:4})
of $L(f \otimes g,s)$.
Our basic object of study is
for any $k,N,\chi$ as above, $m \in \mathbb{N}$
and $g \in S_l(D,\eps)$ the twisted first moment
\begin{equation}\label{eq:3}
  \mathcal{M}_s(k,N,\chi,m,g)
  :=
  \sum_{f \in S_k(N,\chi)}
  \frac{L(f \otimes g,s)}{
    (f,f)
  }
  \overline{\lambda_m(f)},
\end{equation}
where $\lambda_m(f) m^{(k-1)/2}$ is the $m$th
Fourier coefficient of $f$.
Note that
the definition (\ref{eq:3}) is independent of
the choice of orthogonal basis
of $S_k(N,\chi)$.
Letting $S_k^\#(N,\chi)$ denote the subspace
of newforms in $S_k(N,\chi)$ and
$\sum_{f \in S_k^\#(N,\chi)}$ the sum over $f$ in an orthonormal
basis thereof, we similarly define
\begin{equation}\label{eq:5}
  \mathcal{M}_s^\#(k,N,\chi,m,g)
  :=
  \sum_{f \in S_k^\#(N,\chi)}
  \frac{L(f \otimes g,s)}{
    (f,f)
  }
  \overline{\lambda_m(f)}.
\end{equation}

To state our formula for these moments, we must introduce some
notation.
For any integers $a$ and $b$, let $(a,b)$ denote
their greatest common divisor and $[a,b]$ their least common multiple.
For $s \in \mathbb{C}$, let $|.|^s
: \mathbb{N} \rightarrow \mathbb{C}$ denote the multiplicative
function $n \mapsto n^s$.  For a function $\xi : \mathbb{N}
\rightarrow \mathbb{C}$ we let
$\sigma[\xi]$ denote its convolution with the constant function
$1$,
so that
\begin{equation}\label{eq:6}
  \sigma[\xi](n) = \sum_{d | n} \xi(d).
\end{equation}
Here and always the sum is taken over the positive divisors $d$ of $n$.
In the special case that
$\xi$ is a product of a Dirichlet character
with some complex power of $|.|$,
it will be convenient to denote by $\sigma[\xi](0)$ the
value at $s = 0$ of the meromorphic continuation
of the function
$s \mapsto \sigma[\xi |.|^{-s}](0)$,
for which the series definition applies
when $\Re(s) \gg 1$.
For example, we have $\sigma[\chi \eps |.|^{-s}](0) = L(\chi \eps,s)$
for any $s \in \mathbb{C}$.
We extend $\sigma[\xi]$ to negative integers $n$
via $\sigma[\xi](n) = \sigma[\xi](|n|)$.

Now suppose that $\xi$ is a Dirichlet character.
We let $L_N(\xi,s)$ denote the product of its Euler factors
at primes dividing $N$
and $L^N(\xi,s)$ the product of the rest of its Euler factors,
so that for $\Re(s) > 1$,
\[
L(\xi,s) = \sum_{n \geq 1}
\frac{\xi(n)}{n^s}
= L_N(\xi,s) L^N(\xi,s),
\quad L_N(\xi,s)
= \prod_{p|N} (1 - \xi(p) p^{-s})^{-1}.
\]
If the modulus of $\xi$ factors as $\prod m_i$ for some
pairwise relatively
prime positive integers $m_i$, we may write $\xi = \prod
\xi_{m_i}$
where $\xi_{m_i}$ has modulus $m_i$.
For example, our character $\eps$ of modulus $D$
factors as $\eps = \eps_{\delta } \eps_{\delta '}$ for any
factorization $D = \delta \delta '$ and also as
a product $\prod_{p|D}
\eps_p$
over the prime divisors of $D$.  If $\xi$ is primitive of conductor $m$, let
\begin{equation}\label{eq:14}
  \tau(\xi) = \sum_{a \in (\mathbb{Z}/q)^*} \xi(a) e_m(a)
\end{equation}
denote its Gauss sum;
here and always $e_m : \mathbb{Z}/m \rightarrow S^1$ is the
additive character
\begin{equation*}
  e_m(a) = e^{2 \pi i a /m}.
\end{equation*}

Recall that our fixed cusp form $g$ has squarefree level $D$.
Let $\delta$ be a positive divisor of $D$
and $\delta ' = D / \delta$ its complement in $D$.
To each such $\delta$ we associate
the primitive character $\xi$ of conductor $q$
that induces the character $\chi \eps_{\delta '}$
mod $[N, \delta ']$;
to keep the notation uncluttered,
we suppress the dependence on $\delta$
of $\xi$ and $q$.
We adopt the convention that $\xi(0) = 1$ if $q = 1$
and $\xi(0) = 0$ if $q \neq 1$.
For each nonzero integer $A$ we define a factorization
$A = A_1 A_2$ where $0 < A_1 | q^\infty$
(i.e., $A_1$ is a positive product of divisors of $q$)
and $(A_2,q) = 1$;
for $A = 0$, we take $A_1 = 1$ and $A_2 = 0$.
Note that such factorizations depend upon $q$, and hence upon
$\delta$,
but we again suppress this dependence.
Our $\delta$-dependent notation may  be summarized
as
\begin{align*}
  &\delta \delta ' = D,
  \quad \eps = \eps_{\delta } \eps_{\delta '}
  \text{ with }
  \eps_\delta \text{ mod } \delta,
  \eps_{\delta'} \text{ mod } \delta', \\
  &\xi \text{ mod } q \text{ primitive } \rightsquigarrow
  \chi \eps_{\delta '} \text{ mod } [ N, \delta '], \\
  &\xi(0) = 1 \text{ if } q = 1 \text{ and } 0 \text{ otherwise,} \\
  &0 \neq A = A_1 A_2 \text{ with } 0 < A_1 | q^\infty,
  (A_2,q) = 1, \\
  &0 = A = A_1 A_2 \text{ with } A_1 = 1, A_2 = 0.
\end{align*}
For example, the statement of Theorem \ref{thm:finite-formula}
involves an integer $M := [N, \delta ']_2$.
This notation means that $M$ is the largest prime-to-$q$ divisor of
the least common multiple of $N$ and $\delta ' = D/\delta$.
\begin{theorem}\label{thm:finite-formula}
  Let $k,N,\chi,l,D,\eps$ satisfy
  our usual assumptions \eqref{eq:assumptions},
  and let $s \in \mathbb{C}$.
  Suppose that $k > l$,
  that
  $2s$ is an integer of the same parity as $k \pm l$,
  and that $1 - (k-l)/2 \leq s \leq (k-l)/2$.
  Then for each $g \in S_l(D,\eps)$
  and $m \in \mathbb{N}$, we have
  \begin{equation}\label{eq:7}
    \mathcal{M}_s(k,N,\chi,m,g)
    = L(\chi \eps, 2 s)
    \frac{b_m}{m^s}
    + 2 \pi i^{-k}
    \sum_{\delta|D}
    T_s^\delta \mathop{\sum_{n=1}^{m \delta}}_{N_1 | (m \delta - n)_1 q}
    \frac{b_n^\delta }{n^{1-s}}
    I_s \left( \frac{m \delta }{n} \right)
    S_s^\delta(m \delta - n),
  \end{equation}
  where we set $M := [N, \delta ']_2$ and
  \begin{align*}
    T_s^\delta
    &=
    \left( \frac{\delta }{4 \pi ^2 } \right) ^{\tfrac{1}{2} - s}
    \frac{i^l \chi(\delta) \tau(\xi)
      (\eps_\delta |.|^{-2s})(q)
    }{
      \eps_\delta(\delta ')
    }
    % \frac{
    (\eps_\delta \xi |.|^{1-2s})(M)
    % }{
    %   L^M(\eps _\delta \xi, 2 s)
    % }
    \\
    S_s^\delta(x)
    &=
    (\eps_\delta |.|^{1-2s})(x_1)
    \overline{\xi}(x_2)
    \sum_{e|(M,x_2)}
    \mu \left( \frac{M}{e} \right)
    \frac{e}{M}
    \sigma[\eps_\delta \xi |.|^{1-2s}] \left( \frac{x_2}{e}
    \right), \\
    I_s(y)
    &=
    \begin{cases}
      \displaystyle
      y^{\frac{1-k}{2}}
      (y-1) ^{\frac{k-l}{2} - 1 + s }
      \frac{
        \Gamma \left( \frac{l + k}{2} - s \right)
      }{
        \Gamma (l) \Gamma \left( \frac{k - l}{2} + s
        \right)
      }
      F \left(
        \efrac{
          \frac{l - k}{2} + 1 - s,
          \frac{l - k}{2} + s
        }{
          l}
        ; \frac{1}{1 - y}
      \right) 
      & y > 1,
      % \\
      % \Gamma (\frac{l + k }{2 } - s ) \Gamma (2 s - 1 )  /
      % \Gamma (\frac{k - l }{ 2 } + s ) \Gamma (\frac{l - k }{ 2 }+s)
      % \Gamma (\frac{l + k }{2} - 1 + s) & y = 1
      \\
      0 & y = 1, s \neq 1/2,
      \\
      i^{l-k+1}/2
      & y = 1, s =1/2
    \end{cases}
  \end{align*}
  with $F = {}_2 F_1$ the Gauss hypergeometric function.
  If $k - l$ is even, $s = 1$, and for some $\delta \mid D$
  both of the characters $\chi \eps_{\delta'}$ and
  $\eps_{\delta}$
  are principal, then one should interpret the undefined factor
  $I_1(1) S_1^\delta(0) = 0 \cdot \infty$
  in the term $n = m \delta$ of the inner sum
  of the RHS of \eqref{eq:7}
  as the finite quantity
  $(-1)^r \phi(a)/ (2 r a)$,
  where $r = \frac{k-l}{2} - 1 \in \mathbb{Z}_{\geq 0}$
  and $a = \delta M = \delta [N, \delta ']$.
\end{theorem}
\begin{remark}
  Under the hypotheses of
  Theorem \ref{thm:finite-formula},
  the hypergeometric function appearing in the definition
  of $I_s(y)$ for $y > 1$ is a rational function of $y$
  with rational coefficients (that may
  be expressed in terms of associated Legendre
  functions of the first kind).
  Precisely, if $(x)_n = x(x+1) \dotsb (x+n-1)$ is the
  Pochhammer symbol
  and we set $r = \frac{k-l}{2} - s \in \mathbb{Z}_{\geq 0}$,
  then
  \begin{equation}\label{eq:8}
    F \left(
      \efrac{
        \frac{l - k}{2} + 1 - s,
        \frac{l - k}{2} + s
      }{
        l}
      ; x
    \right)
    = \sum_{n=0}^{r}
    \frac{(-r+1-2s)_n (-r)_n}{(l)_n (1)_n}
    x^n.
  \end{equation}
  For example, if $k - l = 1$ and $s = 1/2$,
  then (\ref{eq:8}) is identically $1$.
  Thus the RHS of (\ref{eq:7})
  is an explicit finite expression whose computation
  reduces to that of Dirichlet $L$-values
  and Gauss sums.
\end{remark}
\begin{remark}
  Theorem \ref{thm:finite-formula} gives a formula for the twisted first
  moment over a basis of all cusp forms, not just newforms.
  % One could recover an average over a basis for the space of
  % newforms by an inductive procedure
  % as in \cite{ILS00}.
  In certain applications we have
  $\mathcal{M}_s^\#(k,N,\chi,m,g)=\mathcal{M}_s(k,N,\chi,m,g)$;
  for example,
  this
  happens if there are no oldforms in $S_k(N,\chi)$,
  or if for any oldform $f \in S_k(N,\chi)$ the functional
  equation
  implies
  $L(f \otimes g,1/2)=0$.
  In that case Theorem
  \ref{thm:finite-formula}
  provides (tautologically) an average over newforms.
\end{remark}
\begin{remark}
  When $S_k(N,\chi)$ is one-dimensional and spanned by $f$,
  Theorem \ref{thm:finite-formula}
  gives an exact formula for some values of $L(f \otimes g,s)$.
  In general, one can use linear algebra to recover exact
  finite formulas for
  $L(f \otimes g, s) / (f,f)$ from those given for the twisted
  moments
  $\mathcal{M}_s(k,N,\chi,m,g)$
  for a sufficiently large set of integers $m$.
  This may be of independent computational interest,
  as it gives an exact approach for calculating
  the algebraic part of
  $L(f \otimes g,s)$ in contrast to a more traditional approach
  using
  an approximate
  functional equation.
\end{remark}
\begin{remark}
  One can obtain a similar finite formula with an additional
  explicit term
  when $g$ is non-cuspidal, but for technical reasons we have
  not
  carried this out (see \S\ref{sec:some-results-goldf-2}).
\end{remark}

\subsection{Examples}\label{sec:examples}
In this section we spell out some ready-to-use
deductions
of Theorem \ref{thm:finite-formula}
in special cases for which less notational
overhead is required.
Recall the definition of $I_s$ from the statement
of Theorem \ref{thm:finite-formula}.
\begin{corollary}\label{thm:ex2}
  Suppose that $k$ is odd, $N > 1$, $\chi$ is
  primitive, $l$ is even, $D = 1$, and $k > l$.
  Let $g \in S_l(D) = S_l(1)$.  Then
  \begin{equation}\label{eq:1aaa}
    \mathcal{M}_{1/2}^\#(k,N,\chi,m,g)
    =
    L(\chi,1)
    \frac{ b_m }{ m ^{1/2} }
    +  \frac{2 \pi i^{l-k} \tau(\chi) }{N}
    \sum_{n=1}^{m-1} \frac{b_n}{n^{1/2}} I_{1/2} \left( \frac{m}{n}
    \right) \sigma[\chi](m-n).
  \end{equation}
\end{corollary}
\begin{proof}
  We apply Theorem \ref{thm:finite-formula}.
  Under the stated conditions, we have
  \[
  T_{1/2}^1 = 
  % \frac{
  i^{l}\tau(\chi)
  % }{
  N
  % L(\chi,1)
  % }
  ,
  \quad
  S_{1/2}^1(0) = 0,
  \quad S_{1/2}^1(x) = \sigma[\chi](x)
  \text{ for } x
  \neq 0.
  \]
  Note also that $b_n^1 = b_n$.
  The primitivity
  of $\chi$ implies that there are no oldforms in $S_k(N,\chi)$,
  so $\mathcal{M}_s^\# (\dotsb) = \mathcal{M}_s(\dotsb)$ and
  the claim follows.
\end{proof}
\begin{remark}
  When $m = 1$, the sum over $n$ in (\ref{eq:1aaa})
  is empty, so we recover the statement of
  Theorem \ref{thm:simple}.
\end{remark}

\begin{corollary}\label{thm:ex1}
  Suppose that $k$ is even, $N$ is prime, $\chi = \chi_0$ is the
  principal character
  modulo $N$,
  $l$ is odd, $D$ is prime, $\eps$ is primitive,
  $k > l$, $N \neq D$
  and $g \in S_l(D,\eps)$ is a Hecke eigenform with
  $b_1 = 1$.  Then 
  \begin{eqnarray}\label{eq:13}
    \mathcal{M}_{1/2}(k,N,\chi,m,g)
    &=&
    \left( 1 - \frac{\eps(N)}{N} \right)
    L(\eps,1)
    \frac{b_m}{m^{1/2}}
    + \eps(N) \frac{\tau(\eps)^2 \overline{b_D^2}}{D}
    \left( 1 - \frac{1}{N} \right)
    L(\overline{\eps},1)
    \frac{\overline{b_m}}{m^{1/2}} \\
    &+& \nonumber
    \frac{2 \pi i^{l-k}
      \tau(\overline{\eps})
      \eps(N)
    }{
      D
    }
    ( S_1 + S_D),
  \end{eqnarray}
  where
  \begin{align*}
    S_1 &= 
    \sum_{n=1}^{m-1}
    \frac{b_n}{n^{1/2}}
    I_{1/2} \left( \frac{m}{n} \right)
    \left(
      \overline{\eps}(N)
      \sigma[\overline{\eps}] \left( \frac{m - n }{N} \right)
      -
      \frac{\sigma[\overline{\eps}](m-n)}{N}
    \right), \\
    S_D &= -
    D^{1/2} b_D
    \sum_{n=1}^{mD-1}
    \frac{\overline{b_n}}{n^{1/2}}
    I_{1/2} \left( \frac{m D}{n} \right)
    \left(
      \sigma[\eps] \left( \frac{m - n }{N} \right)
      -
      \frac{\sigma[\eps](m-n)}{N}
    \right)
  \end{align*}
\end{corollary}
\begin{proof}
  The formulas in Theorem \ref{thm:finite-formula} show that
  \[
  T_{1/2}^1 =
  % \frac{
    i^{l} \tau(\eps)
    \eps(N)
  % }{
    D
  %   L^N(\eps,1)
  % }
  ,
  \quad
  S_{1/2}^1(0)
  = 0,
  \quad
  S_{1/2}^1(x) = \overline{\eps}(N) \sigma[\overline{\eps}]
  \left(\frac{x}{N}\right) -
  \frac{\sigma[\overline{\eps}](x)}{N}
  \text{ for }x \neq 0,
  \]
  \[
  T_{1/2}^D =
  % \frac{
    i^{l} \eps(N)
  % }{
  %   L^N(\eps,1)
  % }
  ,
  \quad
  S_{1/2}^D(0) = \left( 1 - \frac{1}{N} \right) L(\eps,0),
  \quad
  S_{1/2}^D(x) = \sigma[\eps]\left(\frac{x}{N}\right) - \frac{\sigma[\eps](x)}{N}
  \text{ for }x
  \neq 0,
  \]
  and $I_{1/2}(1) =
  i^{l-k+1}/2$; in deducing the above formula
  for $S_{1/2}^1(x)$ from that given by Theorem \ref{thm:finite-formula},
  we have used that
  $\sigma[\eps](x) = \sigma[\eps](x y)$ if $y$ is a product of
  factors of $D$
  and that $\sigma[\eps](x) = \eps(x)  \sigma[\overline{\eps}](x)$.
  Moreover, since
  $g$ is an eigenform with primitive nebentypus,
  a formula of Atkin-Li \cite{MR508986}
  (as quoted in Theorem \ref{thm:eta-value})
  shows that
  $b_n^D = - \tau(\eps) D^{-1/2} \overline{b_D b_n}$.
  The terms indexed by $n$ in our formula (\ref{eq:7})
  for which $n \neq m D$ contribute the second
  line of the claimed formula (\ref{eq:5}), while $n = mD$
  contributes the second term on the RHS of \eqref{eq:5}
  after applying the functional equation
  $L(\eps,0) = (\pi i)^{-1} \tau(\eps) L(\overline{\eps},1)$
  and evaluating $i^{2l-2k} = -1$.
\end{proof}

\begin{remark}\label{rmk:whatever}
  For applications
  in which it is known in advance that $L(f
  \otimes g,1/2) = 0$ for all forms $f \in S_k(N)$ that come
  from a form of lower level
  (cf. Corollary \ref{cor:realDihedral}), the related Theorem \ref{thm:real-simplification}
  (in which $N$ is not required to
  be prime) may be more useful.
\end{remark}

\begin{corollary}\label{thm:ex3}
  Suppose that $k$ is odd, $N$ is prime,
  $\chi$ is primitive, $l$ is even,
  $D$ is prime,
  $\eps = 1$,
  $k > l$, $N \neq D$,
  and $g \in S_l(D,\eps)$ is a Hecke eigenform
  with $b_1=1$.
  For a nonzero integer $A$
  write $A = A_1 A_2$ with $A_1$ a power of the prime $N$
  and $(A_2,N) = 1$,
  and write $C = 2 \pi i^{l-k} \tau(\chi) \chi(D) / N$
  for brevity.
  Then
  \begin{eqnarray*}
    \mathcal{M}_{1/2}^\#(k,N,\chi,m,g) &=&
    \left( 1 - \frac{\chi(D)}{D} \right)
    L(\chi,1)
    \frac{b_m}{m^{1/2}} \\
    &+& \nonumber
    C
    \sum_{1 \leq n < m}
    \frac{b_n}{n^{1/2}}
    P_{1/2} \left( \frac{m}{n} \right)
    \overline{\chi}((m-n)_2)
    \left[ \sigma[\chi] \left( \frac{(m - n)_2}{D} \right)
      - \frac{1}{D}
      \sigma[\chi]((m-n)_2)
    \right] \\
    &+& \nonumber
    C
    \sum_{1 \leq n < m D}
    \frac{\mu(D) b_{n D} D^{1/2}}{n^{1/2}}
    P_{1/2} \left( \frac{m D}{n} \right)
    \overline{\chi}((mD-n)_2)
    \sum_{
      \substack{
        d|(mD-n)_2 \\
        (d,D)=1
      }
    }
    \chi(d).
  \end{eqnarray*} 
\end{corollary}
\begin{proof}
  Follows from Theorem \ref{thm:finite-formula}
  in the same manner as
  the previous
  two corollaries, using that
  $b_n^D = \mu(D) b_{n D} D^{1/2}$ when
  $D$ is squarefree and $\eps = 1$ \cite[Proposition
  14.16]{MR2061214}.
  Here we have $\mu(D) = -1$
  because $D$ is assumed to be prime.
\end{proof}

\subsection{Numerical verification}
In this section we carry out some simple
numerical tests of our formulas by choosing
triples $(k,N,\chi)$ for which $\dim S_k(N,\chi) = 1$, say
$S_k(N,\chi) = \langle f \rangle$, so that if $L(f \otimes g, s)
\neq 0$, then Theorem \ref{thm:finite-formula} implies
\begin{equation}\label{eq:9}
  \overline{\lambda_m(f) / \lambda_1(f)}
  =
  \frac{
    \displaystyle
    L(\chi \eps, 2 s)
    \frac{b_m}{m^s}
    + 2 \pi i^{-k}
    \sum_{\delta|D}
    T_s^\delta \sum_{n=1}^{m \delta}
    \frac{b_n^\delta }{n^{1-s}}
    I_s \left( \frac{m \delta }{n} \right)
    S_s^\delta(m \delta - n)
  }
  {
  \displaystyle
    L(\chi \eps, 2 s)
    b_1
    + 2 \pi i^{-k}
    \sum_{\delta|D}
    T_s^\delta \sum_{n=1}^{\delta}
    \frac{b_n^\delta }{n^{1-s}}
    I_s \left( \frac{\delta }{n} \right)
    S_s^\delta(\delta - n)
  }
\end{equation} 
Because the RHS of \eqref{eq:9} contains the terms $b_m/m^s$ and
$b_1$ which pop out immediately in the proof (from the diagonal
term in the Petersson formula), this seems like a
reasonable check of the correctness of Theorem
\ref{thm:finite-formula}.  The computations below were performed
with the computer algebra package SAGE \cite{sage2009}.

\begin{example}
  Take $(l,D,\eps) = (7,7,\chi_{-7})$.  Then $S_l(D,\eps)$ is
  three-dimensional with a Hecke basis $\{g_1,g_2,g_3\}$ where
  $g_1 = q + 9 q^2 + 17 q^4 - 343 q^7 - \dotsb$ is self-dual
  with rational Fourier coefficients and $g_2 = \overline{g_3} =
  q - 8 q^2 + (\alpha+8) q^3 + (-\alpha-8) q^5 + (-8 \alpha -
  64)q^6 + \dotsb$ have coefficients in an imaginary quadratic
  extension of $\mathbb{Q}$, where $\alpha$ satisfies $\alpha^2
  + 16 \alpha + 2104 = 0$.  Take $(k,N,\chi) = (8,3,1)$.  Then
  $S_8(3,1) = S_8(3)$ is one-dimensional, spanned by the
  newform $f = \sum a(n) q^n = q + 6 q^2 - 27 q^3 - 92 q^4 + 390
  q^5 - 162 q^6 - \dotsb$.  Let $\nu_i(m)$ ($i=1,2,3$) denote
  the right hand side of \eqref{eq:13} when $g =
  g_i$.  It is easy to compute $\nu_i(m)$ numerically; we find
  that $\nu_1(1), \dotsc, \nu_1(6)$ are approximately $2.243$,
  $1.189$, $-1.295$, $-1.612$, $3.130$ $-0.686$ and that
  $\nu_2(1), \dotsc, \nu_2(6)$ are approximately $0.362 +
  0.861i$, $0.192 + 0.456i$, $-0.209 - 0.497i$, $-0.260 -
  0.619i$, $0.505 + 1.202i$, $-0.110 - 0.263i$, and $\nu_3(m) =
  \overline{\nu_2(m)}$.  Now set $a_i(n) = n^{(k-1)/2}\nu_i(m) /
  \nu_i(1)$.  Since $S_k(N)$ is one-dimensional, we should have
  $a_i(n) = a(n)$ if our formula is correct.  Indeed, we find
  that $a_i(1),\dotsc, a_i(6)$ are $1.000, 6.000, -27.000,
  -92.000, 390.000, -162.000$ for each $i$; in fact, $a_i(n)$
  and $a(n)$ differ by less than $10^{-13}$ according to our
  computation using floating point precision.  Of course we have
  neglected here the important fact that the $\nu_i(m)$ are
  explicit sums with terms in a cyclotomic extension of
  $\mathbb{Q}$ (up to a multiple of $\pi$).  In this
  case we have
  $(f,f)^{-1} L(f \otimes g_1,1/2) = \nu_1(1) = \pi
  (648/2401)\sqrt{7}$, for example.
\end{example}

\begin{example}
  Let $g \in S_2(11)$ be the weight two cusp form corresponding to the elliptic curve of conductor $11$,
  thus $g = q - 2q^2 - q^3 + 2q^4 + q^5 + 2q^6 - \dotsb$.
  Take $(k,N,\chi) = (5,7,\chi_{-7})$, so that
  $S_5(7,\chi_{-7})$ is one-dimensional and spanned by
  $f = \sum a(n) q^n = q + q^2 - 15 q^4 + 49q^7 - 31 q^8 + \dotsb$.
  Let $\nu(m)$ denote the right hand side of the formula
  given by Corollary \ref{thm:ex1}.
  Our formula then asserts that
  $(f,f)^{-1} L(f \otimes g,1/2) = \nu(1)$.
  Computing numerically, we find
  that $\nu(1) = \pi (6/121) \sqrt{7}$
  and that the numbers $\nu(1), \dotsc, \nu(8)$ are approximately
  $0.412$, $0.103$, $0.000$, $-0.386$, $-3.330$, $1.942$,
  $0.412$, $-0.199$
  with normalized ratios $a_1(n) = n^{(k-1)/2} \nu(n) / \nu(1)$
  approximately $1.000$, $1.000$, $0.000$,
  $-15.000$, $0.000$, $0.000$, $49.000$, $-31.000$,
  in fact that $a_1(n) = a(n)$ to within $10^{-13}$
  (but of course we could have computed this exactly as well).
\end{example}

\begin{example}
  We give an example where $\chi \neq \overline{\chi}$.
  Let $g = q - q^2 - q^4 - 2 q^5 + 4 q^7 + \dotsb \in S_2(17)$ be
  the weight two cusp form corresponding to the elliptic
  curve of conductor $17$, let $\chi$ be the primitive
  character of conductor $11$ for which $\chi(2) = e^{2 \pi i
    /10}$,
  take $(k,N,\chi) = (3,11,\chi)$,
  and let $f$ be the normalized newform that spans the
  one-dimensional
  space $S_3(11,\chi)$.
  If $\zeta = e^{2 \pi i /10}$,
  then we have
  \begin{align*}
    f &= q +
    (-\zeta^3 + 2 \zeta^2 - 2 \zeta)q^2
    + (2 \zeta^3 - 3 \zeta^2 + 3 \zeta - 2)q^3 \\
    &  + (-4 \zeta^2 + 3 \zeta - 4) q^4
    + 4 \zeta^2 q^5 + (\zeta^3 + 3 \zeta^2 - 7 \zeta + 6)
    q^6 + \dotsb \\
    &\approx 1.000
    + (-0.690 - 0.224i)q
    + (-1.118 + 0.812i)q^2
    + (-2.809 - 2.040i)q^3 \\
    &  + (1.236 + 3.804i)q^4
    + (0.954 - 0.310i)q^5
    + (5.854 - 8.057i)q^6
    + \dotsb.
  \end{align*}
  As in the previous examples, let
  $\nu(m)$ denote the right hand side of (\ref{eq:4})
  and $a_1(n) = n^{(k-1)/2} \nu(n) / \nu(1)$.
  Using that $L(\chi,1) =  \pi\sqrt{1/5+i/20}$
  we find that
  $a_1(1), \dotsb, a_1(6)
  \approx 1.000,
  -0.690 + 0.224i,
  -1.118 - 0.812i,
  -2.809+2.040i,
  1.236-3.804i,
  0.954+0.310i,
  5.854+0.8057i$,
  so that $a_1(n) \approx \overline{a(n)}$,
  as expected.
\end{example}

\subsection{Application: stability for twists by mixed signature real
  quadratic theta series}
Let $K / \mathbb{Q}$ be a real quadratic field, $\xi$ a
finite-order Hecke character on $K$, and $g_\xi$ the
corresponding theta series.  Popa \cite{MR2249532} has studied
the Rankin-Selberg $L$-value $L(f
\otimes g_\xi,1/2)$ when $f$ is a newform of trivial central
character and $\xi$ is trivial on $\mathbb{A}_\mathbb{Q}^*$, in
which case $g_\xi$ is a Maass form.  On the other hand, suppose
instead that $\xi$ is a finite order character with mixed
signature at the infinite places $\infty_1, \infty_2$ of $K$, so
that as representations of $\mathbb{R}^*$ we have
$\{\xi_{\infty_1}, \xi_{\infty_2}\} = \{\mathbf{1}, \sgn\}$.
Then $g_\xi$ is a holomorphic cusp form of weight $1$.

\begin{remark}
  One can heuristically explain why $g_\xi$ is holomorphic of
  weight $1$ in the following two ways.  First, the Galois
  representation of $\Gal(\bar{\mathbb{Q}}/\mathbb{Q})$ induced
  by $\xi$ regarded as a character of $\Gal(K^{ab}/K) \cong
  \overline{K^* K_{\infty+}^*} \backslash \mathbb{A}_K^*$ is
  odd (here $K_{\infty+}^*$ is the connected component of the
  identity
  in $(K \otimes_\mathbb{Q}  \mathbb{R})^*$).
  Second, the gamma factor for the $L$-function of $\xi$
  is $\Gamma_\mathbb{R}(s) \Gamma_\mathbb{R}(s+1) =
  \Gamma_\mathbb{C}(s)$, where $\Gamma_\mathbb{R}(s) =
  \pi^{-s/2} \Gamma(s/2)$ and $\Gamma_\mathbb{C}(s) = 2 (2
  \pi)^{-s} \Gamma(s)$.  The gamma factor for a holomorphic cusp
  form of weight $l$ is $\Gamma_\mathbb{C}(\frac{l-1}{2} + s)$,
  while that for a Maass cusp form of eigenvalue $1/4 + \nu^2$
  is $\Gamma_\mathbb{R}(s-\nu) \Gamma_\mathbb{R}(s+\nu)$.  Since
  the lift $g_\xi$ is characterized by the relation $L(\xi,s) =
  L(g_\xi,s)$, we see by comparing the gamma factors that it
  must be holomorphic of weight $l=1$.  See \cite[12.3]{Iw97}
  for further discussion.
\end{remark}

As the following theorem shows, the ``stability'' result for imaginary
quadratic theta series described in the introduction carries
over without essential modification to the mixed signature real
quadratic case.
\begin{theorem} \label{thm:real-concrete}
  Let $K = \mathbb{Q}(\sqrt{D})$ be a real quadratic field of odd
  fundamental discriminant $D > 0$.  Let $\xi$ be a finite order
  Hecke
  character of $K$ of modulus $\mathfrak{m}$ such that with respect
  to a fixed real embedding, we have
  \[
  \xi((\alpha)) = \frac{\alpha}{|\alpha|}\] for all $\alpha \in
  K^*$ with $\alpha \equiv 1 \pmod{\mathfrak{m}}$.  
  Define the cuspidal weight $1$ theta series $g_\xi$ by
  \begin{equation*}
    g_\xi(z) = \sum_{n \geq 1} b_n q^n = \sum_{\mathfrak{a}} \xi(\mathfrak{a}) q^{N_{K/\mathbb{Q}}(\mathfrak{a})},
  \end{equation*}
  where the sum is taken over all nonzero integral ideals
  $\mathfrak{a}$ in the ring $\mathcal{O}_K$ of integers in $K$.
  Suppose that $N_{K/\mathbb{Q}}( \mathfrak{m})$ is squarefree
  and prime to $D$, that the Dirichlet character $n \mapsto
  \xi((n))$ is quadratic of conductor $N_{K/\mathbb{Q}}(
  \mathfrak{m})$, and that the Fourier coefficients $b_n$ are
  real.
  Let $\eps$ be the primitive quadratic Dirichlet character of
  conductor $D \cdot N_{K/\mathbb{Q}}( \mathfrak{m})$ given by
  $\eps(n) = \chi_D(n) \xi((n))$ for all positive integers $n$,
  where $\chi_D$ is the primitive quadratic Dirichlet character associated
  to the extension $K$.
  Let $k$ \RestrictK be even, $N$ a rational prime for which $\eps(-N) =
  1$,
  and $\chi = 1$.
  Then if $N
  > m D N_{K/\mathbb{Q}}(
  \mathfrak{m})$, we have
  \begin{equation} \label{eq:real-concrete}
    \mathcal{M}_{1/2}^\#(k,N,\chi,m,g)
    = 2 \frac{ b_m }{ m^{1/2} } L(\eps,1).
  \end{equation}
\end{theorem}

\begin{remark}
  The quantity on the right hand side of
  (\ref{eq:real-concrete}) has an arithmetic interpretation.
  The Fourier coefficient $b_m$ is the sum of $\xi$ taken over
  the integral ideals in $\mathcal{O}_K$ of norm $m$.  The
  character $\eps$ corresponds to an imaginary quadratic field
  $K'$ of discriminant $D' = -D \cdot N_{K/\mathbb{Q}}(
  \mathfrak{m})$.  Letting $h'$ denote the class number of
  $\mathcal{O}_{K'}$ and $w'$ the cardinality of
  $\mathcal{O}_{K'}^*$, we have
  \[
  L(\eps,1) = \frac{ 2 \pi  h'}{w' |D'|^{1/2} }.
  \]
\end{remark}

\begin{remark}
  In \cite{MR07}, the analogous stability result for imaginary
  quadratic theta series was applied to obtain non-vanishing and
  non-vanishing mod $p$ results for the central Rankin-Selberg
  $L$-values $L(f \otimes g, 1/2)$.  Since the formula
  (\ref{eq:real-concrete}) is identical in our case, the
  applications carry over without modification.
  For instance, with hypotheses as in Theorem
  \ref{thm:real-concrete},
  we find (taking $m = 1$) that
  as soon as $N > D N_{K/\mathbb{Q}}(\mathfrak{m})$,
  there exists
  $f \in S_k(N,\chi)$ for which
  $L(f \otimes g_\xi,\tfrac{1}{2}) \neq 0$;
  under the same conditions,
  if $p > k+1$ and $p$ does not divide $h'$,
  then there exists
  $f \in S_k(N,\chi)$ for which the algebraic
  part of $L(f \otimes g_\xi,\tfrac{1}{2})$
  is not divisible by any place of $\overline{\mathbb{Q} }$
  above $p$.  The former result can be sharpened;
  we refer to \cite[Thm 3, Thm 4]{MR07} for details.
\end{remark}

\subsection{Application: stability in the vertical sense}
We give one example of this new variant of stability,
saving a more general treatment for \S\ref{sec:appl-vert-stab}.
\begin{theorem}\label{thm:stab-vert-sense}
  Suppose that $k$ \RestrictK is even,
  $\chi = 1$,
  $l$ is odd, $\eps$ is primitive,
  $k > l$, and that for some prime divisor
  $p$ of $D$
  we have $p^{\alpha+1} | N$
  and $p^\alpha \geq m D$.
  Then
  \begin{equation}
    \mathcal{M}_{1/2}(k,N,\chi,m,g)
    =
    \frac{b_m}{m^{1/2}}
    L^N(\eps,1).
  \end{equation}
\end{theorem}

\begin{remark}
  In Theorem \ref{thm:finite-formula} and Theorem
  \ref{thm:stab-vert-sense}, we have averaged over an orthogonal basis
  of all cusp forms, including oldforms.  We have been able to
  average over newforms in Theorem \ref{thm:real-concrete} by
  imposing the assumption that $N$ is a prime for which
  $\eps(-N) = 1$: this assumption guarantees that for any
  oldform $f_0$, the sign in the functional equation for $L(f_0
  \otimes g, s)$ is $-1$, whence $L(f_0 \otimes g, \frac{1}{2}) =
  0$.
  
  In some concrete cases, it is possible to apply Theorem
  \ref{thm:stab-vert-sense} directly to compute an average over a basis of
  newforms.  For example, let $\eps$ be the primitive
  character of conductor $3$.  There is a unique cusp form $f$
  (in fact, a newform) of weight $4$, level $9$, nebentypus
  $\eps$ and a unique
  form $g$ of weight $1$, level $3$, nebentypus $\eps$.
  Since $\mathbb{Q}(\sqrt{-3})$ has $6$ units and $1$ ideal
  class, we have $L(\eps,1) = 2 \pi/(6 \sqrt{3})$, so that
  taking $m = 1$ in Theorem \ref{thm:stab-vert-sense} gives
  \begin{equation}
    \frac{L(f \otimes g,\tfrac{1}{2})}{(f,f)}
    = \frac{2 \pi}{6 \sqrt{3}}
  \end{equation}
  (recall our non-standard definition of the Petersson norm
  (\ref{eq:2})).
\end{remark}

\section{Proofs}\label{sec:proofs}
\subsection{The basic idea}\label{sec:basic-idea}
Let $V$ be the finite-dimensional inner product space $V =
S_k(N,\chi)$
with respect to the (scaled) Petersson inner product
\eqref{eq:1}, which we have normalized to be linear in the second
variable.
For any linear functional $\ell : V \rightarrow
\mathbb{C}$, there exists a unique vector $\ell^* \in
S_k(N,\chi)$ (called the kernel of $\ell$) such that $\ell(f) =
(f,\ell^*) = \overline{(\ell^*,f)}$ for all $f \in V$.  Let
$\{f\}$ be an orthonormal basis of $V$.
Then expanding the kernels $\ell_i^* = \sum
\overline{\ell_i(f)} f$, we find $\ell_2(\ell_1^*) =
\overline{\ell_1(\ell_2^*)} = \sum \overline{\ell_1(f)}
\ell_2(f)$.  In particular, defining the expectation
$\mathbb{E}[\overline{\ell_1} \ell_2]$ of the product
$\overline{\ell_1} \ell_2$ by the formula
\begin{equation}\label{eq:10}
  \mathbb{E}[\overline{\ell_1} \ell_2] = \sum \overline{\ell_1(f)} \ell_2(f),
\end{equation}
we obtain:
\begin{lemma}\label{lem:obvious}
  The definition \eqref{eq:10} is independent
  of the choice of orthonormal basis $\{f\}$,
  and satisfies
  \[
  \mathbb{E}[\overline{\ell_1} \ell_2]
  =
  \overline{\ell_1(\ell_2^*)}
  = \ell_2(\ell_1^*).
  \]
\end{lemma}
For example, let $\lambda_m \in V^*$ be
the normalized Fourier coefficient
\begin{equation}
  \lambda_m : V \ni \sum_{n \geq 1} a_n n^{(k-1)/2} q^n
  \mapsto a_m \in \mathbb{C}.
\end{equation}
Then the kernel $\lambda_m^*$ is a multiple of the $m$th
holomorphic
Poincare series, and (for $k \geq 2$) the Petersson formula
expresses
$\mathbb{E} [ \overline{\lambda_m} \lambda_n] =
\lambda_n(\lambda_m^*)
= \overline{\lambda_m(\lambda_n^*)}$
as $\delta_{m n} + \Delta_{m n}$ where $\delta_{m n}$
is the Kronecker delta
and $\Delta_{m n}$ is a sum of Kloosterman sums weighted by
Bessel functions
(see \eqref{eq:delta}).

Now fix a modular form
$g = \sum b_n n^{(l-1)/2} q^n \in S_l(D,\eps)$
and define
$L_s \in V^*$ for $\Re(s) > 1$
by the series
\begin{equation}
  L_s(f) = \sum \frac{\lambda_n(f) b_n}{n^s}
\end{equation}
and for general $s \in \mathbb{C}$ by meromorphic continuation;
then
\begin{equation}
  \mathcal{M}_s(k,N,\chi,m,g)
  = L(\chi \eps, 2 s)
  \mathbb{E} [ \overline{\lambda _m} L_s].
\end{equation}
Our plan in this paper
is to study the moments $\mathbb{E}[\overline{\lambda_m}L_s]$
via the Petersson formula:
for $\Re(s) > 1$, we have
\begin{equation}\label{eq:11}
  \mathbb{E} [ \overline{\lambda_m} L_s]
  = \sum_{n \geq 1} \frac{b_n \mathbb{E} [\overline{\lambda _m }
    \lambda _n ]}{
    n ^s }
  = \frac{b _m }{ m ^s } + \sum _{n \geq 1 } \frac{b _n \Delta _{m n }}{ n ^s }.
\end{equation}
Theorem \ref{thm:finite-formula}
is obtained by writing
\begin{equation*}
  \Delta_{m n} = \int_{a/c \in \mathbb{Q}, w \in \mathbb{C}}
  e_c(n a) n^{-w} \, d \mu_m(a/c,w)
\end{equation*}
for some measure $\mu_m$,
applying Voronoi summation
to $\sum_n b_n e_c(n a) n^{-s-w}$,
and observing that all but finitely many
of the terms in the resulting expression
for $\mathbb{E} [ \overline{\lambda _m } L_s]$
are zero.
From the perspective of the present method,
the key technical point in the proof is that
the function
\begin{equation}\label{eq:26}
\mathbb{C}
\ni w \mapsto
\frac{
  \Gamma \left(
    \frac{k - 1}{2} + w
  \right)
  \Gamma \left(
    \frac{l - 1}{2} + 1-s-w
  \right)
}{
  \Gamma \left(
    \frac{k - 1}{2} + 1 - w
  \right)
  \Gamma \left(
    \frac{l - 1}{2} + s+w
  \right)
},
\end{equation}
which for typical values of $s$
has an infinitude of poles
in the half-plane
$\Re(w) \leq -(k-1)/2$,
happens to be holomorphic in that half-plane
for $k$, $l$ and $s$ satisfying the conditions
of Theorem \ref{thm:finite-formula}
(compare with the proof of \cite[Prop IV.3.3d]{GZ86}).

To put this method in context, we apply Lemma \ref{lem:obvious}
to $\lambda_m$ and $L_s$; doing so gives a relationship
between the $m$th twisted first moment of $f \mapsto L(f \otimes
g,s)$,
the $m$th Fourier coefficient of the kernel of the linear map
$f \mapsto L(f \otimes g,s)$,
and the $L$-value $L(\lambda_m^* \otimes g,s)$:
\begin{lemma}\label{lem:obvious2}
  $\mathbb{E}[\overline{\lambda_m} L_s]
  = \overline{\lambda_m(L_s^*)}
  = L_s(\lambda_m^*)$.
\end{lemma}
When $g$ is the theta series attached to a class group
character of an imaginary quadratic field,
Gross-Zagier \cite[Ch. IV]{GZ86}
compute $\lambda_m(L_{1/2}^*)$
roughly as follows:
the Rankin-Selberg method shows that $L(f \otimes g, s)
= (f, g E_s)_{\Gamma_0(N D) \backslash \mathbb{H}}$
for a non-holomorphic Eisenstein series $E_s$ of level $N D$,
so the kernel $L_{1/2}^*$ is the holomorphic projection
of $\Tr_{N}^{N D}(g E_{1/2})$.

Goldfeld-Zhang \cite{GoZh99}
use the identity $\overline{\lambda_m(L_s^*)}
= L_s(\lambda_m^*)$ of Lemma \ref{lem:obvious2}
and the Petersson formula for $\lambda_n(\lambda_m^*)$
to compute $L_s(\lambda_m^*)
= \sum_n b_n n^{-s} \lambda_n(\lambda_m^*)$
for any $s \in \mathbb{C}$,
with the motivation of giving a simpler
and more general derivation
of Gross-Zagier's result.
The basic idea of writing the Fourier coefficient
of such a kernel function as an infinite linear combination
of Poincar\'{e}  series
had been raised
(but not carried out) by Zagier
\cite{MR0485703}
in the context of $L(\sym^2 f, s)$ (see especially
\cite[p. 38]{MR0485703}).

When $g$ is the theta series attached to a class group
character of an imaginary quadratic field, Michel-Ramakrishnan
used the identity $\mathbb{E} [ \overline{\lambda_m } L_{1/2}]
= \overline{\lambda_m(L_{1/2}^*)}$ of Lemma \ref{lem:obvious2}
and the Gross-Zagier computation of $\lambda_m(L_{1/2}^*)$ to
study the twisted first moments
$\mathbb{E}[\overline{\lambda_m}L_{1/2}]$.

Feigon-Whitehouse \cite{FW08} generalized some of the work of
\cite{MR07} to the case that $f$ lives in a family of
holomorphic Hilbert modular forms over a totally real number
field and $g$ is induced by an idele class character of a CM
extension.  By Waldspurger's formula, they relate $L(f \otimes
g, \tfrac{1}{2})$ to a toral period of the Jacquet-Langlands
correspondent of $f$ on a suitable quaternion algebra, which
they then average by a relative trace formula.

Thus the method of Goldfeld-Zhang applies in the situation we
consider, but for two reasons their results are not directly
applicable here.  First, the only essential assumptions we place
on $k,N,\chi$ and $l,D,\eps$ are that $D$ be squarefree;
Goldfeld-Zhang impose the most restrictive assumptions that
$\chi$ be trivial, $D$ squarefree and $\eps$ primitive.  Second,
and more importantly,
the formulas that they obtain and some of the calculations in
their arguments are not sufficiently detailed for our purposes;
we elaborate on this point in the next section.

\subsection{On some results of Goldfeld-Zhang}\label{sec:some-results-goldf-2}
The formula for $\mathcal{M}_s(k,N,\chi,m,g)$
asserted by Theorem \ref{thm:finite-formula}
follows from a more general result,
which we now prepare to state.
For $s \in \mathbb{C}$ and $y \in \mathbb{R}_+^\times$,
set
\begin{equation}
  I_s(y) =
  \int_{\eps-\frac{k-1}{2}-i \infty }^{\eps-\frac{k-1}{2}+i
    \infty }
  \frac{G_k(w)}{G_l(s+w)}
  y^{-w} \, \frac{d w}{2 \pi i },
  \quad
  G_k(w) := 
  \frac{\Gamma(\frac{k-1}{2} + w)}
  {\Gamma(\frac{k-1}{2}+1-w)}.
\end{equation}
Here the poles of the integrand
(as a function of $w$)
are contained in the sets
$-(k-1)/2 + \mathbb{Z}_{\geq 0}$
and $(l+1)/2 - s + \mathbb{Z}_{\geq 0}$,
so the integral crosses no poles
provided that $\Re(s) < (k+l)/2 - \eps$.
By Stirling in the form $|G_k(w)| \asymp |\Im(w)|^{2\Re(w)-1}$ (uniformly in vertical strips for $w$
away from poles),
the integral converges normally for $1/2 < \Re(s)$,
thereby defining a holomorphic function in the range
$1/2 < \Re(s) < (k+l)/2 - \eps$.
By a contour shifting argument
as in \cite[Proposition
8.3]{GoZh99}, one computes that
\begin{equation}\label{eq:Isoln}
  I_s(y)
  = \begin{cases}
    \displaystyle
    y^{\frac{k-1}{2}}
    (1-y)^{\frac{l-k}{2}-1+s}
    \frac{\Gamma(\frac{l+k}{2}-s)}{\Gamma(k)
      \Gamma(\frac{l-k}{2}+s)}
    F \left(
      \efrac{
        \frac{k-l}{2}+s,
        \frac{k-l}{2}+1-s
      }{
        k
      }
      ; \frac{y}{y-1}
    \right)
    & y \in (0,1) \\
    \displaystyle
    y^{\frac{1-k}{2}}
    (y-1)^{\frac{k-l}{2}-1+s}
    \frac{\Gamma (\frac{l+k}{2}-s)}{\Gamma (l) \Gamma (\frac{k
        - l}{2} + s)}
    F \left(
      \efrac{\frac{l - k }{2}+ 1 - s , \frac{l - k }{2 } +
        s}{l};
      \frac{1}{1 - y}\right) & y > 1 \\
    \displaystyle
    \frac{\Gamma (\frac{l + k }{2 } - s ) \Gamma (2 s - 1 ) }{
      \Gamma (\frac{k - l }{ 2 } + s ) \Gamma (\frac{l - k }{ 2 }+s)
      \Gamma (\frac{l + k }{2} - 1 + s)} & y = 1.
  \end{cases}
\end{equation}
Here the hypergeometric function
$F = {}_2 F _1$ is defined for $|z| < 1$
by the absolutely convergent series
$F(z) = \sum_{n \geq 0} \frac{(a)_n (b)_n}{(c)_n n!} z^n$,
so \eqref{eq:Isoln} gives the meromorphic
continuation of $s \mapsto I_s(y)$ (for fixed $y \in
\mathbb{R}_+^\times$) to the complex plane.
A direct computation with \eqref{eq:Isoln}
shows that the present definition of $I_s(y)$
specializes to that given in the statement of Theorem
\ref{thm:finite-formula};
the key calculation is that for $y = 1$ and
$s$ satisfying the conditions  of
Theorem \ref{thm:finite-formula}, we have $I_s(1) = 0$ unless
$s = 1/2$, in which case
$I_{1/2}(1) = i^{l-k+1}/2$.

The definitions of $T_s^\delta$ and $S_s^\delta(x)$
given in the statement of Theorem \ref{thm:finite-formula}
make sense as meromorphic functions of
$s \in \mathbb{C}$.
The functions $s \mapsto T_s^\delta$
and $s \mapsto S_s^\delta(x)$ for $x \neq 0$ are entire.
We have $S_s^\delta(0) = 0$ (for all $s$)
unless $\xi$ is the trivial character
(equivalently,
unless $\chi \eps_{\delta '}$ is a principal character),
in which case $S_s^\delta(0)$ is a nonzero multiple
of $L(\eps_\delta, 2 s - 1)$,
which is entire with the exception
of a simple pole at $s = 1$ when $\eps _\delta$ is
a principal character.
In order that $\xi$ be trivial and $\eps_\delta$ principal,
we must have in particular $\chi \eps (-1) = 1$,
so that $k \pm l$ is even.
In that case, $I_s(1)$ has a simple zero at $s = 1$
unless $k = l$ thanks to the first or second gamma factor in its
denominator.

The maps $s \mapsto I_s(y)$ are holomorphic
in the half-plane $\Re(s) < k+l$
except for a simple pole
at $s=1/2$ when $y = 1$ and $k \pm l$ is
even.
If $k \pm l$ is even,
then
$S_s^\delta(0)$ fails to vanish to order
at least one at $s = 1/2$ only if
$\chi \eps_{\delta '}$ and $\eps_\delta$ are principal.

It follows that
for $m,n \in \mathbb{N}$, the function
$s \mapsto
c_{m,s}^\delta(n) := I_s(m \delta / n) S_s^\delta(m \delta - n)$
is holomorphic for $\Re(s) < (k+l)/2$
except possibly for simple poles at $s \in \Sigma$
when $n = m \delta$,
where
\begin{equation}\label{eq:18}
  \Sigma_\delta = \begin{cases}
    \{1/2, 1\} & \text{if } k = l \text{ and both }
    \chi \eps_{\delta '} \text{ and }
    \eps_{\delta} \text{ are principal,} \\
    \{1/2\} & \text{if } k \neq  l, k \pm l \text{ is even, and both }
    \chi \eps_{\delta '} \text{ and }
    \eps_{\delta} \text{ are principal,} \\
    \emptyset  & \text{otherwise.}
  \end{cases}
\end{equation}
We turn to the asymptotics
of $c_{m,s}^\delta(n)$ as $n \rightarrow \infty$.
For $y < 1$,
a Taylor expansion shows that
\begin{equation}\label{eq:22}
  I_s(y) =
  y^{\frac{k-1}{2}}
  \frac{\Gamma(\frac{l+k}{2}-s)}{\Gamma(k)
    \Gamma(\frac{l-k}{2}+s)}
  (1 + O(y)),
\end{equation}
where the implied constant is uniform
for $s$ in any fixed compact
subset of the half-plane
$\Re(s) < (k+l)/2$.
For $x \neq 0$,
$|x| \rightarrow \infty$,
we have
\begin{align*}
  \left\lvert
    S_s^\delta(x)
  \right\rvert
  &=
  \left\lvert
    \sum _{e | (M,x_2)}
    \mu \left( \frac{M}{e} \right)
    \frac{e}{M}
    \sigma [\eps_{\delta} \xi |.|^{1-2s}]
    \left( \frac{x_2}{e} \right)
  \right\rvert \\
  &\leq
  \sigma_{|.|^{-1}}(M)
  \sigma [ |.|^{1-2 \Re(s)}]
  (x_2)
  \ll \sigma [ |.|^{1-2 \Re(s)}](x) \\
  &\ll
  |x|^{\max(
    0,
    1 - 2 \Re(s))
    +o(1)},
\end{align*}
where the implied constants is independent
of $s$ and $x$.
Thus as $n \rightarrow \infty$,
we have
\begin{equation}\label{eq:19}
  c_{m,s}^\delta(n) \ll
  n^{-\frac{k-1}{2} + \max(0,1 - 2 \Re(s)) + o(1)},
\end{equation}
where the implied constants
are independent of $n$ and uniform
for $s$ as above.
\begin{lemma}\label{lem:convergence-of-basic-kernel-sum}
  Let $k,N,\chi,l,D,\eps$ satisfy \eqref{eq:assumptions}.
  For $g \in S_l(D,\eps)$, $m \in \mathbb{N}$,
  and $s \in \mathbb{C}$,
  set
  \begin{equation}\label{eq:12}
    c_m(s)
    = 
    L(\chi \eps , 2 s)
    \frac{b_m}{m^s}
    + 2 \pi i^{-k}
    \sum_{\delta|D}
    T_s^\delta \mathop{\sum_{n=1}^{\infty}}_{N_1 | (m \delta - n)_1 q}
    \frac{b_n^\delta }{n^{1-s}}
    I_s \left( \frac{m \delta }{n} \right)
    S_s^\delta(m \delta - n).
  \end{equation}
  The individual terms in the series
  $c_m(s)$ are holomorphic
  functions of $s$
  for $\Re(s) < (k+l)/2$ except
  possibly for simple
  poles in the set $\Sigma
  = \bigsqcup_{\delta|D} \Sigma_\delta$,
  with $\Sigma_\delta$ as in \eqref{eq:18}.
  The series $c_m(s)$ converges normally
  to a holomorphic function
  of $s$ in the domain
  $\Xi =
  \left\{
    \tfrac{3-k}{2} <
    \Re(s) < \tfrac{k-1}{2}
  \right\}
  \setminus \Sigma$.
\end{lemma}
In view of \eqref{eq:19}
and the discussion preceeding it,
the proof of Lemma \ref{lem:convergence-of-basic-kernel-sum}
amounts to noting that
\[
\frac{b_n^\delta }{n^{1-s}}
I_s \left( \frac{m \delta }{n} \right)
S_s^\delta(m \delta - n)
\ll
n^{-1 - \frac{k-1}{2}
  + \max(\Re(s),1 - \Re(s)) + o(1)},
\]
for $n \neq m \delta$,
the RHS of which is $\ll n^{-1-\delta}$
for some $\delta > 0$
if and only if $\max(\Re(s), 1 - \Re(s)) < (k-1)/2$.
\begin{theorem}\label{thm:general-kernel}
  With notation and assumptions as in Lemma
  \ref{lem:convergence-of-basic-kernel-sum},
  we have
  $L(\chi \eps , 2 s)
  \mathbb{E}[\overline{\lambda_m} L_s]
  =    c_m(s)$.
  % L(\chi \eps , 2 s)
  % \frac{b_m}{m^s}
  % + 2 \pi i^{-k}
  % \sum_{\delta|D}
  % T_s^\delta \mathop{\sum_{n=1}^{\infty}}_{N_1 | (m \delta - n)_1 q}
  % \frac{b_n^\delta }{n^{1-s}}
  % I_s \left( \frac{m \delta }{n} \right)
  % S_s^\delta(m \delta - n),
  % \end{equation}
  % where $T_s^\delta$ and $S_s^\delta(x)$
  % are as in Theorem  \ref{thm:finite-formula} and for $\Re(s) > 1/2$,
  % \begin{equation}
  %   I_s(y) =
  %   \int_{\eps-\frac{k-1}{2}-i \infty }^{\eps-\frac{k-1}{2}+i
  %   \infty }
  %   \frac{G_k(w)}{G_l(s+w)}
  %   y^{-w} \, \frac{d w}{2 \pi i },
  %   \quad
  %   G_k(w) := 
  %   \frac{\Gamma(\frac{k-1}{2} + w)}
  %   {\Gamma(\frac{k-1}{2}+1-w)}.
  % \end{equation}
\end{theorem}
We shall prove Theorem \ref{thm:general-kernel}
in \S\ref{sec:proof-theor-refthm:g}.
In the special case that
$\chi = 1$ and $\eps$ is primitive, Theorem
\ref{thm:general-kernel} is similar to what one would obtain by
combining the results \cite[Theorem 6.5]{GoZh99},
\cite[Proposition 7.1]{GoZh99}, and \cite[Proposition
8.3]{GoZh99} of Goldfeld-Zhang
and applying the identity
$L_s(\lambda_m^*) = \mathbb{E}[\overline{\lambda_m} L_s]$
given in Lemma \ref{lem:obvious2}.
The proof we shall give
follows their method, but the result that we obtain differs in
the following ways.
\begin{enumerate}
\item In \eqref{eq:12}, we have restricted the
  sum over $n \geq 1$ by the condition $N_1 | (m \delta - n)_1
  q$, whereas in \cite{GoZh99}, the stronger condition $N | (m
  \delta - n)_1 q$ is imposed (in \cite{GoZh99} one takes $\chi
  = 1$ and $\eps$ primitive, so that $q = \delta '$); it is not
  hard to see from the arguments of \cite{GoZh99} that this
  disagreement is accounted for by a typo.
\item In \cite{GoZh99},
  $i^k$ appears where we have written $i^{-k}$. These are
  equal when $k$ is even, which is the case in \cite{GoZh99};
  $i^{-k}$ is the correct expression when $k$ is odd.
\item The factor $i^{-l}$ appears in \cite{GoZh99} where we have
  written $i^l$; this difference is more serious (giving the
  wrong answer in the important case that $l$ is odd) and results from the propagation
  of a sign error, as we shall explain in \S\ref{sec:proof-theor-refthm:g}.
\item In the proof of \cite[Proposition 8.3]{GoZh99}
  the transformation identity for the hypergeometric function
  is misapplied, resulting in a formula for $I_s(y)$, $y>1$ with ``$k$'' passed as
  the third argument to the hypergeometric function rather
  than ``$l$'' as we have written it above.
\item When $g$ is non-cuspidal, a certain non-entire function
  ``$L_g(s,a/c)$''
  whose poles are determined in \cite[Proposition 4.2]{GoZh99}
  is claimed to be entire in the middle of \cite[p738]{GoZh99};
  some formal manipulations that follow are not valid and
  ultimately yield a formula that is missing a term under
  certain circumstances (including the important case that $k=2$,
  $l=1$, and $g$ is the theta series attached to the trivial
  class group character of an imaginary quadratic extension).
  We limit ourselves to the case that $g$ is cuspidal
  so as not to have to worry about this.

\item    In \cite{GoZh99},
  the analogue of Theorem \ref{thm:general-kernel} appears
  without any restriction on $k \in \mathbb{N}$ or $s \in
  \mathbb{C}$; see for instance the remarks following
  \cite[Prop 3.6]{GoZh99}, where it is asserted
  without proof that
  the analogue of the
  series \eqref{eq:12} converges absolutely
  for \emph{all} values of $s$.
  A closer inspection, as given above
  and summarized in Lemma
  \ref{lem:convergence-of-basic-kernel-sum},
  seems to reveal that such
  series
  converge absolutely only for $s$ in the range
  $(3-k)/2 < \Re(s) < (k-1)/2$.
  When $k = 2$, this range is empty,
  so it is not clear how to interpret the results of \cite{GoZh99}.
  Even when $k = 3$, it seems unclear how to make precise
  certain formal manipulations in \cite{GoZh99}
  without substantially different arguments.
  We record a careful derivation of
  Theorem \ref{thm:general-kernel}
  for $k \geq 4$ in \S\ref{sec:proof-theor-refthm:g}.
\item At the beginning of \cite[\S7]{GoZh99},
  our calculations indicate that one should take
  ``$\eps^\delta  = \eps$'' rather than
  ``$\eps^\delta = \eps_\delta ^{-1} \eps_{\delta '}$''
  for most of the formulas that follow to be correct;
  thus, most instances of $\eps_{\delta '}$ in our formula
  for $S_s^\delta$
  appear as $\eps_{\delta '}^{-1}$
  in \cite[\S7]{GoZh99}.
  (The definition of $S^\delta(s,B)$
  in \cite[\S6]{GoZh99} should have
  ``$\eps_{\delta '}^{-1}(\bar{r})$,'' where $\bar{r}$
  denotes the inverse of $r$ modulo $c$ (a multiple of $\delta
  '$),
  rather than ``$\eps_{\delta '}^{-1}(r)$'' as is written;
  this follows directly from the fact that
  ``$L_g(s+w,\bar{r}/c)$'' appears
  in \cite[6.2]{GoZh99} and that
  ``$L_g(s,a/c) = \eps_{\delta '}^{-1}(a) \times \dotsb$''
  by \cite[Prop 4.2]{GoZh99}.
  To confuse matters further, there is a typo
  in the statement of
  \cite[Lemma 5.3]{GoZh99}: $\eps_{\delta '}$ should be replaced
  by $\eps_{\delta '}^{-1}$ in the definition of the Gauss sum.
  However, it seems that a corrected form of that lemma is
  what is actually
  used in the rest of the paper.)
  To avoid confusion, we carry out the relevant calculations
  in \S\ref{sec:proof-theor-refthm:g}.
\end{enumerate}

\begin{proof}[Proof that Theorem \ref{thm:general-kernel} implies Theorem \ref{thm:finite-formula}]
  Suppose, in
  addition to what is already assumed in Theorem
  \ref{thm:general-kernel}, that $2 s_0$ is an integer of the same
  parity as $k \pm l$, that $k > l$, and that
  $1 - \tfrac{k-l}{2} \leq
  s_0 \leq \tfrac{k-l}{2}$.
  If $n > m \delta$ in the definition of $c_m(s)$, then
  $s \mapsto S_s^\delta(m \delta - n)$ is entire,
  while
  $s \mapsto I_s(m \delta / n)$ is the product of an entire
  function
  with
  $s \mapsto
  \Gamma( \tfrac{l + k}{2} - s )
  / \Gamma( \tfrac{l - k}{2} +  s )$, the latter of which
  vanishes at $s = s_0$.
  Therefore the infinite sum over $n \in \mathbb{N}$ in the definition of
  the series $c_m(s_0)$
  truncates to a finite sum over $n \leq m \delta$.
  To complete the proof of Theorem \ref{thm:finite-formula},
  it remains only to show that
  \begin{equation}\label{eq:20}
    L(\chi \eps, 2 s_0) \mathbb{E}[\overline{\lambda_m} L_{s_0}] =
    c_m(s_0).
  \end{equation}

  The above conditions imply $\Sigma = \emptyset$, so
  the series \eqref{eq:12} defining
  $c_m(s)$ converges normally on
  the strip
  $\Xi = \{s :
  1 - \tfrac{k - 1}{2} < \Re(s) < \tfrac{k - 1}{2}\}$.
  If $s_0 \in \Xi$, then \eqref{eq:20}
  is true by Theorem \ref{thm:general-kernel}.
  The only case in which $s_0 \notin \Xi$
  is if $l = 1$ and either $s_0 = 1 - \tfrac{k -
    1}{2}$
  or $s_0 = \tfrac{k - 1}{2}$,
  so that $s_0$
  is on the boundary of the strip of absolute convergence of the
  series $c_m(s)$.
  Since $s \mapsto L(\chi \eps, 2 s)
  \mathbb{E}[\overline{\lambda_m} L_{s}]$
  is continuous at $s = s_0$, Theorem \ref{thm:general-kernel}
  reduces our claim \eqref{eq:20}
  to showing that
  $\lim_{s \rightarrow s_0} c_m(s) = c_m(s_0)$,
  or even that
  \begin{equation}\label{eq:21}
    \lim_{s \rightarrow s_0}
    (s-s_0)
    \mathop
    {
      \sum_{n > m \delta}
    }
    _
    {
      \substack
      {
        N_1 | (m \delta - n)_1 q \\
      }
    }
    \frac{
      b_n^\delta
    }{
      n^{1-s}
    }
    I_s
    \left(
      \frac
      {
        m \delta
      }
      {
        n
      }
    \right)
    S_s^\delta
    (
    m \delta - n
    )
    = 0
  \end{equation}
  for each $\delta \mid D$.
  Here all limits $s \rightarrow s_0$ are taken over $s$ tending
  to $s_0$ from inside the nonempty open strip $\Xi$.
  While \eqref{eq:21} is formally true,
  a blind interchange of the limit $s \rightarrow s_0$
  with summation over $n$ in \eqref{eq:21}
  is not permitted.
  % , and in fact \eqref{eq:21} can be false
  % for non-cuspidal $g$ (leading to an additional explicit term,
  % compare with \cite[\S4]{MR0485703}).
  Using the Taylor expansion \eqref{eq:22}
  and inserting the definition of $S_s^\delta(m \delta - n)$,
  we reduce \eqref{eq:21}
  to showing that
  for each $\delta \mid D$ and $e \mid [N, \delta ']_2$,
  \begin{equation}\label{eq:23}
    \lim_{s \rightarrow s_0}
    (s-s_0)
    \mathop
    {
      \sum_{n=1}^{\infty}
    }
    _
    {
      \substack
      {
        N_1 | x_1 q \\
        e | x_2 \\
      }
    }
    \frac{
      b_n^\delta
    }{
      n^{1-s + (k-1)/2}
    }
    (\eps _\delta |.|^{1-2s})(x_1)
    \bar{\xi}(x_2)
    \sigma [ \eps_\delta \xi |.|^{1-2s} ]
    \left( \frac{x_2 }{ e}
    \right)
    =0,
  \end{equation}
  where for brevity we write $x = m \delta - n$.
  We consider only the case $s_0 = 1 - \tfrac{k - 1}{2}$,
  as the case $s_0 = \tfrac{k - 1}{2}$
  may be treated in an identical manner thanks to the identity
  $\sigma[\eta](t) = \sigma[\eta ^{-1}](t) \eta(t)$,
  which holds for all multiplicative functions $\eta : \mathbb{N}
  \rightarrow \mathbb{C}$
  and all $t \in \mathbb{N}$.
  Thus, let us write $s = 1 - \tfrac{k-1}{2} + \eps$, where
  $\eps$ is small, positive, and tending to $0$.
  By partial summation, \eqref{eq:23} follows
  from a bound of the shape
  \begin{equation}\label{eq:24}
    \mathop
    {
      \sum_{n \sim X}
    }
    _
    {
      \substack
      {
        N_1 | x_1 q \\
        e | x_2 \\
      }
    }
    b _n ^\delta
    \frac{\sigma [ \eps _\delta \xi \lvert . \rvert ^{k - 2 - 2 \eps
      }](x_2/e)}{n^{k-2-2 \eps}}
    (\eps _\delta |.|^{1-2s})(x_1)
    \bar{\xi}(x_2)
    \ll X^{1-\alpha}
  \end{equation}
  for some fixed $\alpha > 0$ and each $X > m \delta$,
  where $n \sim X$ means $X \leq n < 2 X$.
  Here we require the implied constant to be uniform in $X$ and
  $\eps$, but not necessarily
  in any of the other variables.
  Any of the standard methods for treating shifted convolution
  problems (see e.g. \cite[\S4.4]{MR2331346})  
  applies here to establish \eqref{eq:24} for each $\alpha < 1/2$,
  thereby completing the deduction of Theorem
  \ref{thm:finite-formula}
  from Theorem \ref{thm:general-kernel}
  (and indeed, the ``spectral method'' for doing so
  is implicit in Gross--Zagier's approach
  to an analogous issue in their arguments).
  For completeness, we sketch a direct proof,
  albeit one that suffers the defect of relying
  on our assumption $k > 2$.
  For notational simplicity, we restrict to the case $N = D = e = q = 1$,
  so that the desired bound reads
  $\sum_{n \sim X} b_n \frac{\sigma_{a}(m -
    n)}{n^{a}} \ll X^{1-\alpha}$ for $X > m$, $a = k-2-2 \eps$, and $\eps$
  small and positive,
  where we write
  $\sigma_a(n) = \sigma[|.|^a](n)$ to reduce clutter;
  the general case follows by applying the argument below
  to bound separately the sum over $n$ in each
  arithmetic
  progression modulo $N D$.
  Our assumption $k > 2$ implies that $a$ is bounded from below
  by a fixed positive
  number provided that $\eps$ is taken sufficiently small.
  By opening the divisor sum and applying the hyperbola
  method, 
  we obtain
  \begin{equation}\label{eq:25}
    \begin{split}
      \sum_{n \sim X} b_n \frac{\sigma_a(m-n)}{n^a}
      &=
      \sum_{d_1 < 2 X^{1/2}} d_1^{a}
      \sum_{\substack{
          n \sim X \\
          n \equiv m (d_1)
        }
      }
      \frac{b_n}{n^a} \\
      &+
      \sum_{d_2 < 2 X^{1/2}} d_2^{-a}
      \sum_{\substack{
          n \sim X \\
          n \equiv m (d_2)
        }
      }
      \frac{(n-m)^a}{n^a} b_n
      - 
      \mathop{
        \sum_{d_1 < 2 X^{1/2}} 
        \sum_{d_2 < 2 X^{1/2}}
      }_{d_1 d_2 + m \sim X}
      d_1^{a}
      \frac{b_{d_1 d_2 + m}}{(d_1 d_2 + m)^a}.
    \end{split}
  \end{equation}
  The uniform bound
  $\sum_{n \sim t} b_n e(\alpha n) \ll_g t^{1/2} \log t$
  for all $t \geq 2$ and $\alpha \in \mathbb{R}/\mathbb{Z}$
  (\cite[Thm 5.3]{Iw97})
  implies
  that
  $\sum_{n \sim t, n \equiv m(d)} b_n \ll_g t^{1/2} \log t$
  uniformly for $d \in \mathbb{N}$.
  Inserting this estimate into \eqref{eq:25}
  and summing by parts, we deduce that
  $\sum_{n \sim X} b_n \frac{\sigma_a(m-n)}{n^a}
  \ll X^{o(1) + \max(1-a/2, 1/2)}$,
  as desired.
  % by noting that
  % $I_s(y) = 0$ for $0 < y < 1$ and $s$ satisfying the conditions
  % of Theorem \ref{thm:finite-formula}: (FIXME)
  % indeed, if $r = \frac{k-l}{2} - s$ is a nonnegative
  % integer,
  % then for $y \in (0,1)$ and $2 \Re(z)
  % < k + l$, the function $z \mapsto I_z(y)$ is the product of an
  % entire function with $z \mapsto \Gamma(\frac{l-k}{2} + z)^{-1}$,
  % the latter of which vanishes at $z = s$
  % because of the condition we have imposed on $k-l-2s$.
\end{proof}

Suppose now that $\chi = 1$ and $\eps$ is primitive.
Goldfeld-Zhang simplify
their analogue of
Theorem \ref{thm:general-kernel}
for $L_s(\lambda_m^*)$
under progressively
restrictive assumptions: first that $\eps^2 = 1$, and then that
$(N,D) = 1$ and $g$ is an imaginary quadratic theta series.
We shall make analogous simplifications, but cannot directly
use their results: some of the calculations in their arguments are not sufficiently
detailed for our purposes, and the expression they ultimately
derive for $L_s(\lambda_m^*)$ when $g$ is an imaginary quadratic
theta series is inconsistent with the functional equation
satisfied by $L(f \otimes g, s)$ for $f \in S_k(N) = S_k(N,1)$.
For
completeness, we prove the following variant of
\cite[Proposition 9.1]{GoZh99}.  Recall that $\lambda_m^*$ is
the kernel of the linear functional $\lambda_m$ on the inner
product space $S_k(N)$.  Let us now write
$\lambda_{m,N}^*$ to indicate its dependence on the level $N$.
We say that a form $f \in S_k(N)$ is
\emph{of strictly lower level} if there exists a proper
divisor $\otherN$ of $N$ such that
$f$ is in the image of the tautological inclusion
$S_k(\otherN) \hookrightarrow S_k(N)$
(compare with the Remark at the end of \cite[\S4.1]{GZ86}
and the discussion preceeding it).
\begin{theorem}\label{thm:real-simplification}
  With assumptions as in Theorem \ref{thm:general-kernel},
  suppose also that $\chi = 1$, $\eps$ is primitive, and
  $(N,D) = 1$.  Then there exists a linear combination $v \in
  S_k(N)$ of forms of strictly lower level such that
  \begin{eqnarray} \nonumber 
    L^N(\eps,2s)
    L_s \left( \lambda_{m,N}^* + v \right)
    &=& \frac{b_m}{m^{s}} L(\eps,2s)
    + 2 \pi i^{l-k} (\eps |.|^{1-2s})(N) 
    \sum_{\delta|D} \left( \frac{ \delta }{ 4 \pi ^2 } \right) ^{ 1/2-s}
    \frac{
      \tau(\eps_{\delta '})
    }{
      ( \delta ' ) ^{ 2 s }
    }
    \\ \label{eq:main2}
    &\cdot& 
    \mathop{\sum _{n\geq
        1}}_{N|(m \delta - n)} \frac{ b _n
      ^\delta }{ n ^{ 1 - s } } I _s \left( \frac{ m \delta }{ n } \right)
    ( \eps _\delta | . | ^{ 1 - 2 s } ) [ ( m \delta - n ) _1 ]
    \overline{\eps _{ \delta ' }} [ ( m \delta - n ) _2 ]
    \\ \nonumber &\cdot& 
    \sigma[ \eps
    |.|^{1-2s} ] \left( \frac{ ( m \delta - n ) _2 }{ N } \right).
  \end{eqnarray}
\end{theorem}

\begin{remark}
  \label{rmk:why-is-this-theorem-relevant}
  The RHS of Theorem \ref{thm:real-simplification}
  simplifies substantially in the ``stable range''
  $N > m D$ because of the strong divisibility condition
  in the sum appearing on the second line.
  Under certain assumptions (cf. Remark \ref{rmk:whatever}),
  this leads to analogous simplification
  in the formulas for certain first moments
  in the spirit of item (II) of the introduction.
  See Corollary \ref{cor:realDihedral} for a concrete
  example.
\end{remark}

\begin{proof}
  If $\otherN$ is a divisor of $N$, then we shall regard
  $\lambda_{m,\otherN}^* \in S_k(\otherN)$ as an oldform of lower level in $S_k(N)$.
  Since $(N,D) = 1$ and $\chi = 1$, we have $N_2 = N$ and $N_1 =
  1$ for all divisors $\delta$ of $D$, so we may rewrite
  the assertion of Theorem \ref{thm:general-kernel} as
  \begin{equation} \label{eq:Labc1}
    L_s(\lambda_{m,N}^*) = \frac{ b_m}{m^s} + \sum_{\delta | D}
    B_\delta(N) \sum_{e|N} \mu \left( \frac{ N }{ e } \right) F(\delta,e)
  \end{equation} 
  where (using that $\eps$ is primitive)
  \begin{align} \label{eq:ABC} \nonumber
    A(\delta) &= 2 \pi i^{l-k} \left( \frac{ \delta }{ 4 \pi ^2 }
    \right)^{1/2-s}
    \frac{ \tau( \eps_{\delta '}) }{ (\delta ')^{2s}}, \\ \nonumber
    B_\delta(N) &= B(N) = \frac{ (\eps |.|^{-2s})(N)}{L^N(\eps,2s)}, \\
    C(\delta,n) &= \frac{ b_n^\delta }{n^{1-s}} I_s \left( \frac{ m
        \delta }{ n } \right) ( \eps _\delta |.| ^{ 1 - 2 s } ) [ ( m
    \delta - n ) _1 ]
    \overline{\eps _{ \delta ' }} [ ( m \delta - n ) _2 ] ,
    \\ \nonumber
    E(\delta,n,e) &= e \, \sigma[{\eps |.|^{1-2s}}] \left( \frac{ ( m \delta - n ) _2 }{ e
      } \right) \quad ( := 0 \text{ unless } e | (m \delta - n)_2),
    \\ \nonumber
    F(\delta,e) &= \sum_n A(\delta) C(\delta,n) E(\delta,n,e).
  \end{align}
  We have suppressed the dependence of these expressions on the
  variables $s$ and $m$, which we now regard as fixed;
  what's key is that
  $B_\delta(N) = B(N)$ does not depend upon $\delta$.
  Rearranging the sums in (\ref{eq:Labc1}), we find
  \[
  L_s(\lambda_{m,N}^*) = \frac{ b_m}{m^s} + B(N) \sum_{e|N} \mu \left(
    \frac{ N}{e} \right) \sum_{\delta | D} F(\delta,e).
  \]
  Applying inclusion-exclusion, we obtain
  \begin{eqnarray} \label{eq:Lold}
    L_s \left( \sum_{\otherN|N} \frac{ B(N)}{B(\otherN)} \lambda_{m,N}^*
    \right)
    &=&
    \frac{ b_m}{m^s} \sum_{\otherN|N} \frac{ B(N)}{B(\otherN)}
    + B(N) \sum_{\delta|D} \sum_{\otherN|N} \sum_{e|\otherN} \mu \left( \frac{ \otherN}{e} \right)
    F(\delta,e) \\
    \nonumber &=&
    \frac{ b_m}{m^s} \sum_{\otherN|N} \frac{ B(N)}{B(\otherN)}
    + B(N) \sum_{\delta|D} F(\delta,N).
  \end{eqnarray}

  Let us temporarily set $\psi = \eps |.|^{-2s}$.  Then
  \[
  \sum_{\otherN|N} \frac{ B(N)}{B(\otherN)} = \psi(N) L_N(\eps, 2 s) \sum_{\otherN|N} \frac{ \prod_{p|\otherN} (1
    - \psi(p))}{\psi(\otherN)}.
  \]
  Applying inclusion-exclusion once again, the inner sum is
  \begin{eqnarray*}
    \sum_{\otherN|N} \frac{ \prod_{p|\otherN} (1 - \psi(p))}{\psi(\otherN)}
    &=& \sum_{\otherN|N} \psi^{-1}(\otherN) \sum_{e|\otherN} \mu(e) \psi(e)
    = \sum_{e|N} \psi(e) \mu(e) \mathop{\sum_{\otherN|N}}_{\otherN \equiv 0(e)}
    \psi^{-1}(\otherN) \\
    &=& \sum_{e|N} \mu(e) \sum_{\otherN | \frac{N}{e}} \psi^{-1}(\otherN)
    = \sum_{\otherN|N} \psi^{-1}(\otherN) \sum_{e|\frac{N}{\otherN}} \mu(e) \\
    &=& \psi^{-1}(N),
  \end{eqnarray*}
  so that
  \begin{equation} \label{eq:sumMN}
    \sum_{\otherN|N} \frac{B(N)}{B(\otherN)} = L_N(\eps,2 s).
  \end{equation}

  Substituting (\ref{eq:sumMN}) and (\ref{eq:ABC}) in (\ref{eq:Lold})
  gives (\ref{eq:main2}).
\end{proof}

\subsection{Proof of Theorem \ref{thm:general-kernel}}\label{sec:proof-theor-refthm:g}
In this section we generalize and refine the results of
Goldfeld-Zhang \cite{GoZh99},
as explained in \S\ref{sec:some-results-goldf-2},
following their method.  Let $g =
\sum b_n n^{(l-1)/2} q^n \in S_l(D,\eps)$.  The functional
equation for the additive twists of $L(g,s)$ is proved (up to a
sign; see below) in \cite[Proposition 4.2]{GoZh99}, and asserts
the following.  Let $c$ be a positive integer and $a \in
(\mathbb{Z}/c \mathbb{Z})^*$.
Associate to $c$ the
decomposition $D = \delta \delta'$ with $\delta ' = (c,D) > 0$,
and to $a$ the complex number
\begin{equation} \label{eq:eta}
  \eta_{a/c} = \frac{ i^{l}
  }{\eps_{\delta}(\delta ') \eps_{\delta '}(\delta)} 
    \eps_\delta(c) \eps_{\delta '}(a),
\end{equation}
where $\bar{a} = a^{-1} \in (\mathbb{Z}/ c \mathbb{Z})^*$ as usual.
Recall the notation $G_k(w) = \Gamma(\frac{ k - 1 }{ 2 } + w ) / \Gamma (
\frac{ k - 1 }{ 2 } + 1 - w )$
from the statement of Theorem \ref{thm:general-kernel}.
Then
\begin{equation} \label{eq:func-eqn}
  \sum_n \frac{b_n e_c(\bar a n)}{n^{s}}
  = \eta_{a/c} \left( \frac{\delta c ^2}{4 \pi ^2} \right)^{\frac{1}{2}-s}
  \frac{1}{G_l(s)}
  \sum_n \frac{ b^\delta_n e_c(-a \bar \delta n)}{n^{1-s}}.
\end{equation}
More precisely, the series involved in
\eqref{eq:func-eqn}
converge absolutely for $s$ in suitable half-planes
and extend to entire functions of $s$ satisfying \eqref{eq:func-eqn}.
In \cite[Proposition 4.2]{GoZh99} the 
formula (\ref{eq:func-eqn}) is obtained
but with the factor $i^l$ in $\eta_{a/c}$
replaced by $i^{-l}$.  This results from a sign error in
\cite[page 735, displayed equation 6]{GoZh99}, where one should
write $(-c z)^l$ instead of $(cz)^l$.  This propagates to a sign
error in the statement of \cite[Proposition, 4.2]{GoZh99} and
the theorems that follow.  As a check, take $c =
1$, $a = 0$, $\delta = D, \delta ' = 1$.  In that case it is a
classical theorem of Hecke (see for instance \cite[Theorem
14.7]{MR2061214}) that the functional equation relating $g$
(with coefficients $b_n$) to its Fricke involute $g|W_D$ (with
coefficients $b_n^D$) is
\[
\sum_n \frac{b_n}{n^s}
= i^l D^{1/2-s}
\frac{\Gamma(\frac{l-1}{2}+1-s)}{\Gamma(\frac{l-1}{2}+s)}
\sum_n \frac{ b^D_n}{n^{1-s}},
\]
which agrees with (\ref{eq:func-eqn})
as we have written it.  The statement and proof of
\cite[Proposition, 4.2]{GoZh99} are otherwise correct.

The term $\Delta_{m n}$ in
the Petersson formula $\mathbb{E}[\overline{\lambda_m}
\lambda_n] = \delta_{m n} + \Delta_{m n}$ is given by
(see \cite[Proposition 14.5]{MR2061214})
\begin{equation} \label{eq:delta}
  \Delta_{mn} = 2 \pi i^{-k} \mathop{\sum_{c \geq 1}}_{N|c} \frac{S_\chi
    (m,n,c)}{c} J_{k-1} 
  \left(\frac{4 \pi \sqrt{m n}}{c} \right),
\end{equation}
where for any $\eps \in (0,1/2)$
\[S_\chi(m,n,c) = \sum_{a \in (\mathbb{Z} / c \mathbb{Z})^*}
\chi(a) e_c(m a + n \bar{a}),
\quad
J_{k-1}(x) = \int_{(\eps - \frac{k-1}{2})}
G_k(w) \left( \frac{ x }{ 2 }
\right) ^{ - 2w } \,d w
\]
are the Kloosterman sum and Bessel function.
For $\Re(s) > 1$ we have
\begin{equation}\label{eq:plug-in-here}
  \mathbb{E}[\overline{\lambda_m}L_s]
  = \sum_{n \geq 1}
  \frac{b_n \mathbb{E}[\overline{\lambda_m} \lambda_n]}{n^s}
  = \frac{b_m}{m^s} + \sum_{n \geq 1} \frac{b_n \Delta_{m
      n}}{n^s},
\end{equation}
so it remains only to compute $\sum b_n \Delta_{m n} n^{-s}$.
By definition,
\[
\sum_n \frac{b_n \Delta_{mn} }{n^s}
=
2 \pi i^{-k}
\sum_n
\mathop{\sum_{c \geq 1}}_{N|c} \frac{S_\chi (m,n,c)}{c} J_{k-1} 
\left(\frac{4 \pi \sqrt{m n}}{c} \right) \frac{b_n}{n^s}.
\]
The estimate $J_{k-1}(x) \ll_k \min(x^{k-1},x^{-1/2})$
and the Weil bound
$|S_\chi(m,n,c)| \leq (m,n,c) \tau(c) c^{1/2}$
imply
that the double sum over $n$ and $c$ on the RHS converges absolutely
for $\Re(s) > 5/4$.
Interchanging summation and opening the Kloosterman sum
gives
\[
\sum_n \frac{b_n \Delta_{mn} }{n^s}
=
2 \pi i^{-k} \mathop{\sum_{c \geq 1}}_{N|c}
\frac{
  1
}{
  c
}
\sum_n
\sum_{a \in (\mathbb{Z}/c \mathbb{Z})^*}
\chi(a)
  e_c(a m)
e_c(\bar{a} n)
J_{k-1} 
\left(\frac{4 \pi \sqrt{m n}}{c} \right) \frac{b_n}{n^s}.
\]
Here the inner double sum over $n$ and $a$ converges absolutely
for $\Re(s) > 5/4$ (or even $\Re(s) > 3/4$)
thanks to the bound $J_{k-1}(x) \ll_k
x^{-1/2}$,
so we may interchange the $n$ and $a$ sums to get
\begin{equation}\label{eq:17}
\sum_n \frac{b_n \Delta_{mn} }{n^s}
=
2 \pi i^{-k} \mathop{\sum_{c \geq 1}}_{N|c}
\sum_{a \in (\mathbb{Z}/ c\mathbb{Z} )^*}
\frac{
  e_c(a m)
}{
  c
}
A(a/c,s),
\end{equation}
where
\[
A(a/c,s) :=
\sum_n
\chi(a)
e_c(\bar{a} n)
J_{k-1} 
\left(\frac{4 \pi \sqrt{m n}}{c} \right) \frac{b_n}{n^s}.
\]
\begin{lemma}\label{lem:convergence-bla-bla-bla-bla}
  For each $k \geq 2$, the function $A(a/c,s)$, defined initially for $\Re(s) >
  3/4$,
  extends to a meromorphic function on the
  complex plane
  that is holomorphic in the half-plane $\Re(s) > 1/2$.
  Moreover, for $k > 2$ and
  $s$ in the region $\Omega := \{s \in \mathbb{C} : 1/2 <
  \Re(s) < (k-1)/2\}$,
  we have
  \begin{equation}\label{eq:16}
  A(a/c,s)
  =
  \eta_{a/c}
  \left( \frac{\delta c^2}{4 \pi ^2 } \right)^{1/2-s}
  \sum_n
  \frac{
    b_n^\delta e_c(-a \bar{\delta }n)
  }{
    n^{1-s}
  }
  I_s
  \left(
    \frac{m \delta }{n}
  \right),
\end{equation}
where $I_s(y)$ is as in the statement of
Theorem \ref{thm:general-kernel}
and the decomposition $D = \delta \delta '$ is associated
to $c$ as above.
  % where for $s \in \Omega$ and $\eps \in (0,1/10)$,
  % \begin{equation}\label{eq:3}
  %   I(y,s)
  %   :=
  %   \int _{
  %     \left(
  %       \eps - \frac{k - 1 }{2}
  %     \right)
  %   }
  %   \frac{G_k(w)}{G_l(s+w)}
  %   y^{- w}
  %   \, \frac{d w}{2 \pi i }.
  % \end{equation}
\end{lemma}
\begin{proof}
  Formally, this is a direct application of the functional equation
  \eqref{eq:func-eqn}; the only issue
  is in justifying convergence in the
  intermediate steps.
  % (neglected
  % in \cite{GoZh99})
  % This is impora more than a technicality, because neglecting such
  % convergence issues
  % leads to the wrong answer in certain cases, e.g.,
  % in the motivating case considered by
  % Gross--Zagier.)
  The asymptotic $I_s(y) \sim y^{(k-1)/2}$ as $y \rightarrow 0$
  shows that the series on the RHS
  of \eqref{eq:16}
  converges normally on $\Omega$.
  By known properties of
  $I_s(y)$  as summarized at the beginning of \S\ref{sec:some-results-goldf-2},
  we are done if we can show
  that
  $A(a/c,s)$
  extends to the half-plane $\Re(s) > 1/2$
  and that when $k > 2$, $A(a/c,s)$ satisfies \eqref{eq:16} on any nonempty
  open subset of $\Omega$.
  We outline the proof of this last assertion:
  \begin{enumerate}
  \item Inserting the Mellin expansion $J_{k-1}(2x) = \int _{(\eps -
      \frac{k - 1}{2})}
    G_k(w) x^{-w}
    \, \frac{d w}{2 \pi i}$
    gives
    \[
    A(a/c,s)
    = \sum _{n} \int _{\left( \eps - \frac{k - 1}{2} \right)}
    \frac{b _n e _c (\bar{a}n)}{n ^{s + w}}
    G_k(w) \left( \frac{4 \pi ^2 m}{ c ^2 } \right) ^{- w }
    \, \frac{d w}{2 \pi i}.
    \]
    By Stirling in the form $G_k(w) \ll (1 + |\Im( w)|)^{2 \Re(w)
      - 1}$,
    this double sum/integral converges absolutely for $\Re(s) >
    (k+1)/2$.
    Interchanging summation with integration gives
    \begin{equation}\label{eq:15}
      A(a/c,s)
      = \int _{\left( \eps - \frac{k - 1}{2} \right)}
      G_k(w) \left( \frac{4 \pi ^2 m}{ c ^2 } \right) ^{- w }
      D(s + w)
      \, \frac{d w}{2 \pi i}
      \quad  \text{with }
      D(w) :=
      \sum _{n} \frac{b _n e _c (\bar{a }n)}{ n ^w }
    \end{equation}
    for $\eps > 0$ and $\Re(s) > (k+1)/2$.
  \item The convexity bound for $D(w)$,
    obtained by the functional equation, Stirling's formula and
    Phragmen-Lindel\"{o}f, asserts that
    the estimate
    $|D(w)| \ll_\eps (1 + |\Im(w)|)^{\eta(\Re(w)) + \eps}$
    holds uniformly in vertical strips,
    where
    \[
    \eta(\sigma) = 
    \begin{cases}
      1 - 2 \sigma & \sigma \leq 0 \\
      1 - \sigma  & 0 \leq \sigma \leq 1 \\
      0 & 1 \leq \sigma.
    \end{cases}
    \]
    With some case-by-case analysis, it follows
    that the integral \eqref{eq:15} converges normally for
    $\Re(s) \geq 1/2 + 3 \eps$.  Taking $\eps \rightarrow 0$, we
    obtain the claimed holomorphic continuation of $A(a/c,s)$ to
    the half-plane $\Re(s) > 1/2$.
  \item
    Suppose now that $k > 2$.
    We aim now to establish the claimed identity \eqref{eq:16}
    for $s$ in the (nonempty) subregion
    $\Omega_\eps = \{s \in \mathbb{C} : 1/2 + 3 \eps \leq 
    \Re(s) \leq  1/2 + 4 \eps \}$ of $\Omega$.
    To do so, we apply the functional equation
    \eqref{eq:func-eqn} in the form
    \[
    D(s + w)
    = \eta_{a/c} \left( \frac{\delta c ^2 }{4 \pi ^2 } \right)
    ^{1/2- s}
    \frac{\tilde{D}(1-s-w)}{G_l(s+w)},
    \quad
    \tilde{D}(1-s-w) :=
    \sum _n \frac{b _n ^\delta e _c (- a \bar{\delta }n)}{ n ^{1 - s-w}},
    \]
    giving (for $\Re(s) \geq 1/2 + 3 \eps$)
    \[
    A(a/c,s)
    =
    \left( \frac{\delta c^2 }{4 \pi ^2 } \right) ^{1/2 - s}
    \eta_{a/c}
    \int _{\left( \eps - \frac{k - 1}{2} \right)}
    \frac{
      G_k(w)
    }{
      G_l(s+w)
    }
    (m \delta ) ^{- w }
    \tilde{D}(1-s-w)
    \, \frac{d w}{2 \pi i}.
    \]
    Since $s \in \Omega_\eps$ and $k > 2$,
    we have $\Re(s) \leq (k-1)/2 - \eps$,
    i.e., $\Re(1 - s - w) \geq 1+ \eps$.
    Therefore the series definition of $\tilde{D}$ applies.
    The resulting double sum/integral converges
    absolutely,
    so we deduce the claimed identity \eqref{eq:15}
    by interchanging.
  \end{enumerate}
\end{proof}

\begin{remark}
  Using the meromorphic continuation of $s \mapsto I_s(y)$,
  one can make sense of the RHS of \eqref{eq:15}
  in the larger range $\Re(s) < (k-1)/2$.
\end{remark}
In order to insert the result \eqref{eq:16}
of Lemma \ref{lem:convergence-bla-bla-bla-bla}
(valid for $1/2 < \Re(s) < (k-1)/2$)
into the formula \eqref{eq:17}
(valid for $\Re(s) > 5/4$),
the domains of validity of these two assertions should have
nonempty
overlap.  The latter holds if and only if
$5/4 < (k-1)/2$, i.e., if and only if
$k \geq 4$.
Thus, for $k \geq 4$ and $1/2 < \Re(s) < (k-1)/2$,
\[
\sum_n \frac{b_n \Delta_{mn} }{n^s}
=
2 \pi i^{-k}
% \sum_{\delta | D}
% \mathop{\mathop{\sum_{c \geq
% 1}}_{N|c}}_{(c,D)=\delta '}
\sum_{N | c}
\frac{\left( \frac{\delta }{4 \pi ^2 } \right)^{1/2-s}}{c^{2 s}}
\sum_{a \in (\mathbb{Z}/ c\mathbb{Z} )^*}
\chi(a)
\eta_{a/c}
e_c(a (m  - \bar{\delta }n))
\sum_n
\frac{
  b_n
}{
  n^{1-s}
}
I_s
\left(
  \frac{m \delta }{n}
\right).
\]
The calculations below
(in Lemma \ref{lem:ramsum},
of the Gauss--Ramanujan sum over $a$)
show that the
the double sum over $c$ and $n$ converges absolutely
for $s$ in the stated range, so that a final interchange
of summation and a substitution
$a \mapsto \delta a$ gives
\begin{equation}\label{eq:sumDelta}
\sum_n \frac{b_n \Delta_{mn} }{n^s}
=
2 \pi i^{-k} \sum_{\delta | D} \left( \frac{ \delta }{ 4 \pi ^2 }
\right)^{\frac{1}{2} - s} \sum_n \frac{b_n^\delta
}{n^{1-s}}
I_s \left( \frac{ m \delta }{ n } \right)
R_s^\delta(m \delta - n),
\end{equation}
where
% Substituting the definition (\ref{eq:delta}) of
% $\Delta_{mn}$ and applying the functional equation (\ref{eq:func-eqn}),
% we obtain (using here that our additively-twisted
% $L$-functions are entire)
% \begin{eqnarray} \nonumber
%   \sum_n \frac{b_n \Delta_{mn} }{n^s}
%   &=& 
%   2 \pi i^{-k} \mathop{\sum_{c \geq 1}}_{N|c} \sum_n \frac{S_\chi (m,n,c)}{c} J_{k-1} 
%   \left(\frac{4 \pi \sqrt{m n}}{c} \right) \frac{b_n}{n^s} \\
%   \nonumber &=&
%   2 \pi i^{-k} \mathop{\sum_{c \geq 1}}_{N|c}
%   \sum_{a \in (\mathbb{Z} / c \mathbb{Z})^*} \frac{\chi(a)
%     e_c(ma)}{2 \pi i c} \int_{(\eps-\frac{k-1}{2})}
%   G_k(w) \left( \frac{4
%       \pi^2 m}{c^2} \right) ^{-w}
%   \sum_n \frac{b_n e_c(\bar a n)}{n^{s+w}} 
%   \,d w \\
%   \label{eq:sumDelta} &=& 
%   2 \pi i^{-k} \sum_{\delta | D} \left( \frac{ \delta }{ 4 \pi ^2 }
%   \right)^{\frac{1}{2} - s} \sum_n \frac{b_n^\delta
%   }{n^{1-s}}
%   I_s \left( \frac{ m \delta }{ n } \right)
%   R_s^\delta(m \delta - n),
% \end{eqnarray}
% where $I_s(y)$ is as in the statement of Theorem
% \ref{thm:general-kernel}
% and
\begin{align*}
  R_s^\delta(x) &= \mathop{\mathop{\sum_{c \geq
        1}}_{N|c}}_{(c,D)=\delta '} c^{-2s} \sum_{a \in
    (\mathbb{Z} / c \mathbb{Z})^*} \chi(\delta a)
  \eta_{\delta a/c} e_c(x a).
\end{align*}
Substituting the definition (\ref{eq:eta}) of
$\eta_{a/c}$, we find
\begin{equation}\label{eq:Rexpr}
  R_s^\delta(x) = \frac{ i^{l} \chi(\delta) }{\eps_\delta(\delta ') }
  \mathop{\sum_{N|c \geq 1}}_{(c,D)= \delta ' }
  c^{-2s} \eps_\delta(c) \sum_{a \in (\mathbb{Z} / c
    \mathbb{Z})^*} 
  \chi  \eps_{\delta
      '}(a) e_c(x a).
\end{equation}
Recall from the beginning of \S\ref{sec:main-result} that we let $\xi$ denote the
primitive character of conductor $q$ that induces $\chi
\eps_{\delta '} $,
that by convention $\xi(0) = 1$ if $q = 1$
and $\xi(0) = 0$ otherwise,
that for $A \neq 0$ we write $A = A_1 A_2$
with $0 < A_1 | q^\infty$ and $(A_2,q) = 1$, that for $A = 0$ we
take $A_1 = 1$ and $A_2 = 0$,
and that we set $M = [N_2, \delta'_2] = [N,\delta ']_2$.

Since $q$ divides $[N,\delta ']$ which in turn divides
$c$, we have
$\chi \eps_{\delta '}(a)
= \xi (a)$ whenever $(a,c) = 1$,
so that we may write
$\xi$ instead of $\chi \eps_{\delta '}$ in
(\ref{eq:Rexpr}).
\begin{lemma} \label{lem:ramsum}
  If $x_1 q = c_1$, then we have
  \begin{equation} \label{eq:ramsum} \sum_{a \in (\mathbb{Z} / c
      \mathbb{Z})^*} \xi(a) e_c(x a) =
    x_1 \overline{\xi} (x_2)  \xi  (c_2) \tau(\xi ) \sum_{d|(c_2,x_2)} d \mu
    \left( \frac{ c _2 }{ d } \right),
  \end{equation}
  otherwise the sum on the left hand side of (\ref{eq:ramsum}) vanishes.
\end{lemma}

\begin{proof}
  Suppose first that $x \neq 0$.  Let us write $p^\alpha||n$ to
  denote that a prime $p$ divides $n$ to order exactly $\alpha$,
  in which case we also write $\alpha = v_p(n)$.
  The Chinese remainder theorem
  shows that $1 \equiv \sum_{p^\alpha || c} \overline{c p^{-a}} c
  p^{-a} \pmod{c}$, where $\overline{c p^{-a}}$ is any inverse
  mod $p^\alpha$.  Thus
  \[
  e_c(x a) = \prod_{p^\alpha || c} e_{p^\alpha}
  (\overline{cp^{-\alpha}} x a),
  \]
  and we may factor the left hand side of (\ref{eq:ramsum}) as
  \[
  \left( \prod_{p^\alpha || c_1} \sum_{a \in (\mathbb{Z} / p^\alpha \mathbb{Z})^*}
    \xi_p(a) e_{p^\alpha}(\overline{cp^{-\alpha}} x a)
  \right) \sum_{a \in (\mathbb{Z} / c_2 \mathbb{Z})^*} e_{c_2}(x a).
  \]
  Similarly,
  \[
  \tau(\xi) = \prod_{p^\beta  || q}
  \sum_{a \in (\mathbb{Z} / p^\beta \mathbb{Z})^*} \xi _p(a) e_{p^\beta}( \overline{ q  p^{-\beta}} a).
  \]
  The evaluation of the Ramanujan sum
  \[
  \sum_{a \in (\mathbb{Z} / c_2 \mathbb{Z})^*} e_{c_2}(x a) = \sum_{d|(c_2,x_2)} d \mu \left(
    \frac{ c_2 }{ d } \right)
  \]
  is well-known: by Mobius inversion, this amounts to the
  assertion that $\sum_{a \in \mathbb{Z} / n \mathbb{Z}}
  e_{n}(x a)$ is $n$ if $n | x$, $0$ otherwise, and that
  $(c_2,x) = (c_2,x_2)$.

  Fix a prime divisor $p$ of $c_1$ (equivalently of $q$, since $q$ and
  $c_1$ have the same support), and write $\alpha = v_p(c_1)$,
  $\beta = v_p(q)$.  Since $\xi_p$ is primitive of
  conductor $p^\beta \neq 1$ and the additive character $a \mapsto
  e_{p^\alpha}(\overline{ c p^{-\alpha}} x a)$ has period
  $p^{\alpha - v_p(x)}$, we see that $\sum_{a \in (\mathbb{Z} /
    p^\alpha \mathbb{Z} )^*} \xi_p(a)
  e_{p^\alpha}(\overline{ cp^{-\alpha}} x a)$ vanishes unless
  $v_p(x) = \alpha - \beta$, in which case
  \begin{eqnarray} \label{eq:sumpalpha}
    \sum_{a \in (\mathbb{Z} / p^\alpha \mathbb{Z})^*} \xi_p(a)
    e_{p^\alpha}(\overline{ cp^{-\alpha}} x a)
    &=& p^{\alpha-\beta} \sum_{a \in (\mathbb{Z} / p^{\beta}
      \mathbb{Z})^*} 
    \xi_p(a)
    e_{p^\beta} \left( \overline{ cp^{-\alpha}} \frac{x}{p^{\alpha-\beta}} a \right) \\ \nonumber
    &=& \xi_p \left( \frac{ c }{ p^{\beta} x } \right)
    p^{\alpha - \beta} \sum_{a \in (\mathbb{Z} / p^{\beta}
      \mathbb{Z})^*}
    \xi_p(\overline{
      q p^{-\beta}} a) e_{p^\beta}(\overline{ q p^{-\beta}} a) \\ \nonumber
    &=& \xi_p \left( \frac{ c }{ q x} \right) 
    p^{\alpha - \beta} \sum_{a \in (\mathbb{Z} / p^{\beta}
      \mathbb{Z})^*} 
    \xi_p(a)
    e_{p^\beta}(\overline{ q p^{-\beta}} a).
  \end{eqnarray}
  In the second step we made the substitution
  $(\mathbb{Z}/p^\beta \mathbb{Z})^* \ni a \mapsto c \overline{x q} a$.
  
  Suppose that (\ref{eq:sumpalpha}) is nonzero for each $p^\alpha || c_1$.
  Then $v_p(x)= v_p(c_1) - v_p(q)$ for each $p|c_1$, so that
  $x_1 q = c_1$, and we obtain
  \begin{eqnarray*}
    \sum_{a \in (\mathbb{Z} / c \mathbb{Z})^*}
    \xi(a) e_c(x a)
    = x_1 \xi  \left( \frac{ c_2}{x_2}
    \right) \tau( \xi) \sum_{d | (c_2,x_2)} d
    \mu \left( \frac{ c_2 }{ d} \right),
  \end{eqnarray*}
  as desired.  This completes the proof of Lemma
  \ref{lem:ramsum} when $x \neq 0$.  In the remaining case that
  $x = 0$, the sum (\ref{eq:ramsum}) vanishes unless $q = 1$, in
  which case it equals $\sum_{d |c_2} d \mu \left( \frac{c_2}{d}
  \right)$.  Given our conventions,
  this is what the lemma asserts.
\end{proof}

By (\ref{eq:ramsum}), we have $R_s^\delta(x) = 0$ unless there
exists a positive integer $c$ such that $N_1 | c_1$ and $c_1 =
x_1 q$, so let us assume henceforth that $N_1 | x_1 q$.  Then
\begin{eqnarray}\label{eq:Sdelta}
  R_s^\delta(x) &=&
  \frac{ i^{l}\chi(\delta) }{\eps_\delta(\delta ') \eps_{\delta '}^2(\delta)}
  \mathop{\sum_{N|c \geq 1}}_{(c,D)= \delta ' }
  (\eps_\delta |.|^{-2s})(c) \sum_{a \in (\mathbb{Z} / c
    \mathbb{Z})^*} 
  \xi(a) e_c(x a) \\ \nonumber
  &=& \frac{i^{l}\chi(\delta)}{\eps_\delta(\delta ') \eps_{\delta '}^2(\delta
    )}
  \sum_{
    \substack{
      N_2| c_2 \geq 1 \\
      (c_2,D) = \delta _2'
    }
  }
  ( \eps_\delta |.|^{-2s})(x_1 q c_2)
  x_1 \overline{\xi} (x_2)
  \xi (c_2)
  \tau(\xi )
  \sum_{d|(c_2,x_2)}
  d \mu \left( \frac{c_2}{d} \right) \\ \nonumber
  &=&
  \frac{
    i^{l} \chi(\delta)\tau(\xi)
    ( \eps_\delta |.|^{-2s})(q) (\eps_\delta |.|^{1-2s}) (x_1)
    \overline{\xi}(x_2)
  }{ \eps_\delta (\delta ') \eps_{\delta '}^2(\delta)}
  \sum_{M|c_2 \geq 1}
  (\eps_\delta \xi |.|^{-2s})(c_2)
  \sum_{d|(c_2,x_2)}
  d \mu \left(\frac{c_2}{d}\right). 
\end{eqnarray}
We evaluate the sums over $c_2, d$ in the following lemma,
where we set $\psi = \eps_\delta
\xi |.|^{-2s}$.
\begin{lemma} Provided that the following sums converge
  absolutely, we have
  \begin{equation}\label{eq:doublesum}
    \sum_{M|c_2 >0} \psi(c_2) \sum_{d|(c_2,x_2)} d \mu \left( \frac{ c_2 }{ d }
    \right)
    = \frac{ (\psi |.|^1)(M)}{L^M(\psi)} \sum_{e|(M,x_2)} \mu \left( \frac{ M
      }{ e } \right) \frac{ e }{ M } \sigma[{\psi |.|^1}] \left( \frac{
        x_2}{e} \right).
  \end{equation}

\end{lemma}

\begin{proof}
  Interchanging summation, we have
  \[
  S := \sum_{M|c_2} \psi(c_2) \sum_{d|(c_2,x_2)} d \mu \left( \frac{ c_2 }{ d }
  \right)
  = \sum_{d|x_2} d \sum _{[M,d]|c_2} \psi(c_2) \mu \left( \frac{ c_2 }{ d }
  \right),
  \]
  where $[a,b]$ is the least common multiple of $a$ and $b$.
  The maps
  \[
  c_2 \mapsto \left( \frac{ M d }{ (c_2 , M d ) } , \frac{ c_2 }{ ( c_2, M d )
    } \right),
  \quad
  (e,l) \mapsto \frac{ M d }{ e} l,
  \]
  give a bijection of sets of natural numbers
  \[
  \{ c_2 : [M,d] | c_2 \}
  \leftrightarrow
  \{
  (e,l) : e | (M,d), (l,e) = 1,
  \}
  \]
  so that
  \begin{eqnarray*}
    S &=& \sum_{d|x_2} d \sum_{e|(M,d)} \sum_{(l,e)=1} \psi \left( \frac{ M
        d }{ e} l \right) \mu \left( \frac{ M }{e} l \right)
    = \sum_{d|x_2} d \sum_{e|(M,d)} \psi \left( \frac{ M d}{e} \right)
    \mu \left( \frac{ M }{ e} \right)
    \mathop{\sum_{(l,e)=1}}_{(l,\frac{M}{e})=1} \psi(l) \mu(l) \\
    &=& \frac{1}{L^M(\psi)} \sum_{e|M} \sum_{d| \frac{x_2}{e}} d e \,
    \psi(M) \psi(d) \mu \left( \frac{ M }{ e} \right)
    = \frac{ (\psi |.|^1)(M) }{ L^M(\psi)} \sum_{e|M} \mu \left(
      \frac{ M}{e} \right)  \frac{ e}{M} \sigma[{\psi |.|^1}]
    \left( \frac{ x_2 }{ e
      } \right),
  \end{eqnarray*}
  as desired.
\end{proof}

Taking $\psi = \eps_\delta \xi |.|^{-2s}$ and
substituting (\ref{eq:doublesum}), (\ref{eq:Sdelta}),
(\ref{eq:sumDelta}) into (\ref{eq:plug-in-here}) gives
Theorem \ref{thm:general-kernel}
in the range $5/4 < \Re(s) < (k-1)/2$,
and hence in the stated range $(k-3)/2 < \Re(s) < (k-1)/2$ by
analytic continuation.
\begin{remark}
  There are other choices of coefficients $(b_n )$ that lead to
  interesting linear functionals $L_s = \sum b_n n^{-s} \lambda_n$.
  For instance, one could let $b_n$
  be the $n$th Fourier coefficient
  of a Maass cusp form,
  or one could take $b_n = \tau(n)$,
  so that $\zeta^{N}(2 s) L_s(f) = L(f,s)^2$
  for $f \in S_k(N)$.
  In the latter case, the functional equations for
  the additive twists \[\sum_{n \geq 1} \tau(n) e^{2 \pi i
    \alpha n} n^{-s}, \quad \alpha \in \mathbb{Q} \]
  follow (see e.g.
  \cite{KMV02})
  from the expansions
  \[
  \frac{ \partial }{ \partial s } E ( z , s ) |_{s = 1/2}
  = y^{1/2} \log y + 4 y^{1/2} \sum_{n \geq 1} \tau(n) K_0(2 \pi n y)
  \cos(2 \pi n x),
  \]
  \[
  \frac{-1}{2 \pi } \frac{\partial^2}{\partial x \partial s} E ( z , s ) |_{s = 1/2}
  = 4 y^{1/2} \sum_{n \geq 1} n \tau(n) K_0(2 \pi n y)
  \sin(2 \pi n x),
  \]
  where
  $E(z,s)$ is a
  real-analytic Eisenstein series for $SL(2,\mathbb{Z})$.
  However, an inspection of
  the proof of Theorem
  \ref{thm:finite-formula}
  shows that one does
  \emph{not} obtain a finite formula
  for $\mathbb{E}[\overline{\lambda_m} L_s]$
  in such cases
  (cf. the discussion surrounding \eqref{eq:26}).
  % does not simplify
  % to a finite sum when $s = 1/2$, but rather to an infinite
  % series half of whose terms vanish.  In general, when the
  % coefficients $b_n$ are those of some $GL_2/\mathbb{Q}$
  % automorphic form, the finiteness of such a formula following
  % this approach boils down to an assertion about the locations
  % of the zeros and poles of the Gamma factor in the
  % functional equation for the additive twists of $\sum b_n
  % n^{-s}$.
  For this reason, we have restricted our attention to
  the case that the $b_n$ are Fourier coefficients
  of holomorphic forms.
  % However, tracing through the proof of Theorem
  % \ref{thm:finite-formula} given above,
  % one sees that the resulting formula
  % for $\mathbb{E}[\overline{\lambda_m} L_s]$ does not simplify
  % to a finite sum when $s = 1/2$, but rather to an infinite
  % series half of whose terms vanish.  In general, when the
  % coefficients $b_n$ are those of some $GL_2/\mathbb{Q}$
  % automorphic form, the finiteness of such a formula following
  % this approach boils down to an assertion about the locations
  % of the zeros and poles of the Gamma factor in the
  % functional equation for the additive twists of $\sum b_n
  % n^{-s}$.  For this reason we have restricted our attention to
  % holomorphic forms $g$.
\end{remark}

\subsection{Proofs of applications}
Recall the notation
\eqref{eq:3}, \eqref{eq:5}
for the twisted first moments
$\mathcal{M}_{s}(k,N,\chi,m,g)$,
$\mathcal{M}_{s}^\#(k,N,\chi,m,g)$.
Recall also the ``$\delta$-dependent'' notation
$\delta ', \eta$ mod $q$, $M$, $A = A_1 A_2$
introduced in \S \ref{sec:main-result}
and used in the statement of Theorem \ref{thm:finite-formula}.
\subsubsection{Real dihedral twists}
\label{sec:appl-real-dihedr}
\begin{theorem} \label{thm:realDihedral}
Preserve the assumptions \eqref{eq:assumptions}.
  Suppose that $k$  is
  even, $\chi = 1$, $l$ is odd, $\eps$ is primitive quadratic,
  $k > l$,
  and
  $(N,D) = 1$.  Suppose moreover that $g$ is a
  Hecke eigenform with $b_1 = 1$ and that $N,k$
  have been chosen so that for any form $f \in S_k(M) \subset
  S_k(N)$
  of strictly lower level $M | N$, $M \neq N$, we have
  $L(f \otimes g,\frac{1}{2}) = 0$.  For brevity
  write $x = m \delta - n$.
  Then
  \begin{eqnarray} \label{eq:dihedral-general}
    \mathcal{M}_{1/2}^\#(k,N,\chi,m,g)
    &=&
    \frac{b_m}{m^{1/2}}
    L(\eps,1)
    + \eps(-N)
    \overline{b_D^2}
    \frac{\overline{b_{m}}}{m^{1/2}}
    L(\eps,1) \\ \nonumber
    &+&
    2 \pi i^{l-k} \eps(N) \sum_{\delta|D}
    \frac{\tau(\eps_{\delta '})}{\delta '}
    \sum_{
      \substack{
        1 \leq n < m \delta \\ \nonumber
        N | (m \delta - n)
      }
    }
    \frac{b_n^\delta}{n^{1/2}}
    P_{1/2} \left( \frac{m \delta }{n} \right)
    \eps_\delta(x_1)
    \eps_{\delta '}(x_2)
    \sigma[\eps] \left( \frac{x_2}{N} \right).
  \end{eqnarray}
  Here $b_n^D = - i \overline{b_{n D}}$.
\end{theorem}

\begin{proof}
  Under the given conditions, we have $L_{1/2}(v) = 0$ for all
  oldforms $v \in V$,
  so that
  Theorem \ref{thm:real-simplification}
  implies that $\mathcal{M}_{1/2}^\#(k,N,\chi,m,g)$ is equal to
  the sum of 
  \[
  \frac{b_m}{m^{1/2}} L(\eps,1)
  + 2 \pi i^{-k}
  \eps(N)
  \frac{i^l b_{m D}^D}{(m D)^{1/2}}
  I_{1/2}(1) L(\eps,0)
  \]
  and the second line of (\ref{eq:dihedral-general}).
  We have $I_{1/2}(1) =
  i^{l-k+1}/2$ and $b_{m
    D}^D = - \tau(\eps) \overline{b_{m D^2}} D^{-1/2}$
  (see Theorem \ref{thm:eta-value}).
  The functional equation
  $\pi i L(\eps,0) = \tau(\eps) L(\overline{\eps},1)$
  shows that
  \[
  2 \pi i^{-k} \eps(N) \frac{i^l b_{m D}^D}{(m D)^{1/2}}
  I_{1/2}(1) L(\eps,0) = \frac{\tau(\eps)^2}{D} \eps(N)
  \frac{\overline{b_{m D^2}}}{m^{1/2}}
  L(\overline{\eps},1).
  \]
  We have $\tau(\eps)^2 = D \eps(-1)$ and $\overline{\eps} = \eps$ since $\eps$ is primitive
  quadratic, so the formula (\ref{eq:dihedral-general})
  follows.
\end{proof}

\begin{corollary} \label{cor:realDihedral} With conditions
  and assumptions as in Theorem \ref{thm:realDihedral},
  suppose furthermore that $N > m D$.
  Then
  \begin{equation} \label{eq:dihedral}
    \mathcal{M}_{1/2}^\#(k,N,\chi,m,g)
    =
    \frac{b_m}{m^{1/2}}
    L(\eps,1)
    + \eps(-N)
    \overline{b_D^2}
    \frac{\overline{b_{m}}}{m^{1/2}}
    L(\eps,1).
  \end{equation}
\end{corollary}
\begin{proof}
  The sum over $n$ in (\ref{eq:dihedral-general})
  is empty when $N > m D$.
\end{proof}

\begin{proof}[Proof of Theorem \ref{thm:real-concrete}]
  Let $g = g_\xi \in S_1(D',\eps)$, $D' = D \cdot
  N_{K/\mathbb{Q}}(\mathfrak{m})$ be as in the statement of Theorem
  \ref{thm:real-concrete}.  Then $g$ is a normalized Hecke
  eigenform because $\xi$ is a character and a newform because
  $\xi$ (or $\eps$) is primitive.  Since $\eps$ is primitive and
  $g$ is a normalized newform, Theorem \ref{thm:eta-value}
  implies that $|b_D| = 1$.
  Since $\eps$ is quadratic, we have
  $\eps = \overline{\eps}$.
  Since $b_D \in \mathbb{R}$, we have $\overline{b_D^2} = 1$.
  Since by assumption $\eps(-N) = 1$,
  we see that Theorem
  \ref{thm:real-concrete} follows from (\ref{eq:dihedral})
  provided that we can justify the hypothesis that
  $L(f \otimes g, 1/2) = 0$ for all forms
  $f \in S_k(N)$ coming from a lower level.
  Since $N$ is prime, the only possibility is
  that $f \in S_k(1)$.
  We may assume by linearity that $f$ is a normalized Hecke eigenform,
  so that in particular $\overline{f} = f$, $\overline{g} = g$.
  Then \cite[Theorem 2.2, Example 2]{Li79}
  shows that
  \begin{equation}\label{eq:6b}
    L(f \otimes g,s)
    = \eps(f \otimes g)
    (1 ^2 D^2)^{1/2-s} \frac{\gamma(1-s)}{\gamma(s)}
    L(f \otimes g, 1- s),
  \end{equation}
  where
  \[
  \gamma(s) = 
  \Gamma_\mathbb{C}
  \left( s + \frac{|k-l|}{2} \right)
  \Gamma_\mathbb{C} 
  \left( s + \frac{k+l}{2}-1 \right),
  \quad
  \Gamma_\mathbb{C}(s) = 2(2 \pi)^{-s} \Gamma(s),
  \]
  and
  \[
  \eps(f \otimes g)
  =
  \begin{cases}
    \eps(-1) \chi(D) \eps(1) \eta(f)^2 \eta(g)^2 &  k \leq l \\
    \chi(-1) \chi(D) \eps(1) \eta(f)^2 \eta(g)^2 & k > l,
  \end{cases}
  \]
  with $\eta(f), \eta(g)$ defined by
  $f|W_1 = (-1)^k \eta(f) \overline{f}$
  and $g|W_D = (-1)^l \eta(g) \overline{g}$;
  here $W_1,W_D$ are the Fricke involutions as defined in
  \S\ref{sec:main-result}.
  Since $\eps$ is primitive quadratic and $b_D \in \{\pm 1\}$, Theorem
  \ref{thm:eta-value}
  implies that $\eta(g)^2 = \eps(-1)$.
  Since $f$ has trivial level, we have $\eta(f) = 1$.
  Since $k > l$ and $\chi = 1$, we see that
  $\eps(f \otimes g) = \eps(-1) = -1$.
  Evaluating the functional equation (\ref{eq:6b}) at the point
  $s = 1/2$ gives $L(f \otimes g,1/2) = 0$, as desired.
\end{proof}

\begin{remark}
  Note that the argument just given applies also to cuspidal
  imaginary quadratic theta series (in which context our 
  calculation of $b_{m D}^D$ remains valid), so we have
  recovered the cuspidal case of the original stability result
  of \cite{MR07}.
\end{remark}

\subsubsection{Vertical stability}
\label{sec:appl-vert-stab}
\begin{theorem}\label{thm:vert-stab}
Preserve the assumptions \eqref{eq:assumptions}.
  Suppose that $k > l$, $k - l \equiv 1 \pmod{2}$, and that $N$
  is chosen so that for each divisor $\delta |D$
  with $(\delta, N) = 1$
  we have $N_1 / (N_1,q) \geq \max(m \delta,2)$.
  Then $\mathcal{M}_{1/2}(k,N,\chi,m,g)
  = L(\chi \eps,1)
  b_m m^{-1/2}$.
\end{theorem}
\begin{proof}
  Fix a divisor $\delta$ of $D$.
  If $(\delta,N) > 1$,
  then $\chi(\delta) = 0$,
  so $T_{1/2}^\delta = 0$.
  If $N_1/(N_1,q) \geq m
  \delta$, then the sum over $n$ in
  \eqref{eq:7}
  is empty
  with the possible exception of the term indexed by $n = m \delta$, which
  vanishes because $N_1 / (N_1,q) \geq \max(m \delta,2)$ implies $q > 1$
  implies $\eta(0) = 0$, so that $S_s^\delta(m \delta - m
  \delta) = 0$.
  Thus the claim follows from
  Theorem \ref{thm:finite-formula}.
\end{proof}

\begin{theorem}\label{sec:vert-stab-2}
Preserve the assumptions \eqref{eq:assumptions}.
  Suppose that $k$ is even, $\chi = 1$,
  $l$ is odd, $\eps$ is primitive,
  and $k > l$.
  Let $(N,D^\infty) = \lim_{\alpha \to \infty}
  (N,D^\alpha)$,
  and suppose that
  \[
  \frac{(N,D^\infty)}{(N,D)} \geq \max(m D,2).
  \]
  Then $\mathcal{M}_{1/2}(k,N,\chi,m,g) =
  L^N(\eps,1) b_m m^{-1/2}$.
\end{theorem}
\begin{proof}
  Let $\delta$ be a divisor of $D$ for which $(\delta,N) = 1$.
  Then $q = \delta'$ and $(N,D) = (N, \delta ') = (N,q)
  = (N_1,q)$, $N_1 = (N,D^\infty)$, so that
  $N_1/(N_1,q) = (N,D^\infty)/(N,D) \geq
  \max(m D,2) \geq \max(m \delta,2)$; the claim then follows from
  the criterion of Theorem \ref{thm:vert-stab}.
\end{proof}

\begin{proof}[Proof of Theorem \ref{thm:stab-vert-sense}]
  The conditions of Theorem \ref{sec:vert-stab-2} are satisfied
  when there exists a prime divisor $p$ of $D$
  and $\alpha \geq 1$ for which
  $p^{\alpha+1} | N$ and $p^{\alpha} \geq m D$.
\end{proof}

% -----------------
\bibliography{refs}{}
\bibliographystyle{plain}
% -----------------
\end{document}